\documentclass{article} 
\usepackage{amsfonts}
\usepackage[english]{babel}
\usepackage{latexsym}
\usepackage[latin1]{inputenc}
\usepackage{amssymb}
\usepackage{amsmath}
\usepackage{bbm}
\usepackage{amsthm}
\usepackage{graphicx}
\usepackage[usenames]{color}

\newcommand{\pd}[2]{\frac{\partial #1}{\partial #2}}
\newcommand{\es}[2]{\left( #1, #2\right)}
\newtheorem{theorem}{Theorem}
\newtheorem{lemma}[theorem]{Lemma}
\newtheorem{definition}[theorem]{Definition}

\title{Finite element schemes for a class of
nonlocal parabolic systems with moving boundaries}
\author{Rui M. P. Almeida\thanks{University of Beira Interior, Department of Mathematics, Faculty of Science, email: ralmeida@ubi.pt, jduque@ubi.pt, rrobalo@ubi.pt}   \and Jos\'e C. M. Duque$^*$  \and Jorge Ferreira\thanks{Federal University Fluminense - UFF - VCE; email: ferreirajorge2012@gmail.com} \and Rui J. Robalo$^*$}

\date{\today}

\begin{document}

\maketitle

\begin{abstract}
The aim of this paper is to establish convergence, properties and error bounds for
the fully discrete solutions of a class of nonlinear systems of
reaction-diffusion nonlocal type with moving boundaries, using the finite element method with polynomial approximations
of any degree. A coordinate transformation which fixes the boundaries is used. Some
numerical tests to compare our Matlab code with a moving finite
element method are investigated.\\


\textbf{keywords}: nonlinear parabolic system; nonlocal diffusion term;
reaction-diffusion;  convergence;  numerical  simulation;  Euler; Crank-Nicolson;
finite element method.
\end{abstract}

\section{Introduction}
\label{intro}
In this work, we study parabolic systems with nonlocal nonlinearity of the
following type:
\begin{equation}  \label{probi}
\left\{
\begin{array}{l}
\displaystyle \frac{\partial u_i}{\partial t}-a_i\left( \int_{\Omega
_{t}}u_1(x,t)dx,\dots,\int_{\Omega _{t}}u_{{n_e}}(x,t)dx\right) \frac{%
\partial^2u_i}{\partial x^2}=f_i\left(x,t\right)\,, \quad (x,t)\in Q_{t} \\
\displaystyle u_i\left( \alpha (t),t\right) =u_i\left( \beta
(t),t\right)=0\,,\quad t>0 \\
\displaystyle u_i(x,0)=u_{i0}(x)\,, \quad x\in \Omega _{0}=]\alpha (0),\beta
(0)[,\quad i=1,\dots,{n_e} \\
\end{array}
\right. \,
\end{equation}
where $Q_{t}$ is a bounded non-cylindrical domain defined by
$$Q_{t}=\left\{ (x,t)\in \mathbb{R}^{2}:\:\alpha(t)<x<\beta(t),\:\: \text{for
all } 0<t<T\right\}\,$$
and
$$\Omega_t=\{x\in\mathbb{R}:\: \alpha(t)<x<\beta(t), 0\leq t\leq T\}.$$
Problem (\ref{probi}) arises in a large class of real models,
for example, in biology, where the solution $u$ could describe the density
of a population subject to spreading; or in physics, where $u$ could represent
the temperature, considering that the measurements are an average in a
neighbourhood \cite{chi00}. It is nonlocal in the sense that the diffusion
coefficient is determined by a global quantity, that is, $a$ depends on the
whole population in the area.

This class of problems, with nonlocal coefficients in an open bounded
cylindrical domain, was initially studied by Chipot and Lovat  in \cite{CL97},
 where they proved the existence and uniqueness of weak solutions. In
recent years, nonlinear parabolic equations with nonlocal diffusion terms
have been extensively studied \cite
{CL99,AK00,CM01,CC03a,CS03,CVC03,CMF04,ZC05}, especially in relation to
questions of existence, uniqueness and asymptotic behaviour.

In order to model interactions, a system is needed. Raposo et al.
\cite{RSVPS08}, in 2008, studied the existence, uniqueness and exponential
decay of solutions for reaction-diffusion coupled systems of the form
\begin{equation*}
\left\{
\begin{array}{lrl}
u_{t}-a(l(u))\Delta u+f(u-v)=\alpha (u-v) & \text{ in } & \Omega \times
]0,T], \\
v_{t}-a(l(v))\Delta v-f(u-v)=\alpha (v-u) & \text{ in } & \Omega \times
]0,T], \\
\end{array}%
\right.
\end{equation*}%
with $a(\cdot )>0$, $l$ a continuous linear form, $f$ a Lipschitz-continuous
function and $\alpha $ a positive parameter. Recently, Duque et al. \cite%
{DAAFppa} considered nonlinear systems of parabolic equations with a more
general nonlocal diffusion term working on two linear forms $l_{1}$ and $%
l_{2}$:
\begin{equation}
\left\{
\begin{array}{lrl}
u_{t}-a_{1}(l_{1}(u),l_{2}(v))\Delta u+\lambda _{1}|u|^{p-2}u=f_{1}(x,t) &
\text{ in } & \Omega \times ]0,T], \\
v_{t}-a_{2}(l_{1}(u),l_{2}(v))\Delta v+\lambda _{2}|v|^{p-2}v=f_{2}(x,t) &
\text{ in } & \Omega \times ]0,T]. \\
\end{array}%
\right.   \label{sisDAAF}
\end{equation}%
They gave important results on polynomial and exponential decay, vanishing
of the solutions in finite time, and localisation properties such as waiting
time effect.

Moving boundary problems occur in many physical applications involving
diffusion, such as in heat transfer where a phase transition occurs, in
moisture transport such as swelling grains or polymers, and in deformable
porous media problems where the solid displacement is governed by diffusion
(see, for example, \cite{fer97,fer99,fer05,bri07,BG04}). Cavalcanti et al \cite%
{fer03}  worked with a time-dependent function $a=a\left( t,\int_{\Omega
_{t}}\left\vert \nabla u(x,t)\right\vert ^{2}dx\right) $ to establish the
solvability and exponential energy decay of the solution for a model given
by a hyperbolic-parabolic equation in an open bounded subset of $\mathbb{R}%
^{n}$, with moving boundary. Santos et al. \cite{san05} established the
exponential energy decay of the solutions for nonlinear coupled systems of
beam equations with memory in noncylindrical domains. Recently, Robalo et
al. \cite{RACF14} proved the existence and uniqueness of weak and strong
global in time solutions and gave conditions, on the data, for these
solutions to have the exponential decay property. The analysis and numerical
simulation of such problems presents further challenges. In \cite{AK00},
Ackleh and Ke propose a finite difference scheme to approximate the
solutions and to study their long time behaviour. The authors also made
numerical simulations,  using an implicit finite difference scheme in one
dimension \cite{RSVPS08} and the finite volume discretisation  in two space
dimensions \cite{EGHM02}. Bendahmane and Sepulveda \cite{BS09}, in 2009,
investigated the propagation of an epidemic disease modelled by a system of
three PDE, where the $i$th equation is of the type
\begin{equation*}
(u_{i})_{t}-a_{i}\left( \int_{\Omega }u_{i}dx\right) \Delta
u_{i}=f_{i}\left( u_{1},u_{2},u_{3}\right) ,
\end{equation*}%
in a physical domain $\Omega \subset \mathbb{R}^{n}$ $(n=1,2,3)$.
They established the existence of solutions for finite volume schemes and their
convergence to the weak solution of the PDE. In \cite{DAAF15}, the authors
proved the optimal order of convergence for a linearised Euler-Galerkin
finite element method for  problem (\ref{sisDAAF}) and presented some
numerical results. Almeida et al., in \cite{ADFR14}, established
convergence, properties and error bounds for the fully discrete solutions of a class of
nonlinear equations of reaction-diffusion nonlocal type with moving
boundaries, using a linearised Crank-Nicolson-Galerkin finite element method
with polynomial approximations of arbitrary degree. In \cite{RACF14}, Robalo et
al. also obtained approximate numerical solutions for equations of this
type with a Matlab code based on the Moving Finite Element Method (MFEM)
with high degree local approximations.

In this paper, we study the convergence of the total discrete solutions using the finite element method with some classical time integrators. To the best of our knowledge, these results are new for nonlocal reaction-diffusion systems with moving boundaries.

The paper is organized as follows. In Section 2, we formulate the problem and the hypotheses on the data. In Section 3, we define and prove the convergence of the semidiscrete solution. Section 4 is devoted to the proof of the existence, uniqueness, stability and convergence of the fully discrete solutions for each method. In Section 5, we obtain and compare the approximate numerical solutions for one example. Finally, in Section 6, we draw some conclusions.

\section{Statement of the problem}

In what follows, we study the convergence of the totally discrete solutions of the
one-dimensional Dirichlet problem with two moving boundaries, defined by
\begin{equation}  \label{prob0}
\left\{
\begin{array}{l}
\displaystyle \frac{\partial u_i}{\partial t}-a_i\left( \int_{\Omega
_{t}}u_1(x,t)dx,\dots,\int_{\Omega _{t}}u_{{n_e}}(x,t)dx\right) \frac{%
\partial^2u_i}{\partial x^2}=f_i\left(x,t\right), \,(x,t)\in Q_{t} \\
\displaystyle u_i\left( \alpha (t),t\right) =u_i\left( \beta
(t),t\right)=0\,,\quad t>0 \\
\displaystyle u_i(x,0)=u_{i0}(x)\,, \quad x\in \Omega _{0}=]\alpha (0),\beta
(0)[,\quad i=1,\dots,{n_e} \\
\end{array}
\right. \,
\end{equation}
where
\begin{equation*}
Q_{t}=\left\{ (x,t)\in\mathbb{R}^{2}:\:\alpha(t)<x<\beta(t),\:\: \text{for
all } 0<t<T\right\}\,
\end{equation*}
is a bounded non-cylindrical domain, $T$ is an arbitrary positive real
number, $a_i$ denotes a positive real function and
$$\Omega_t=\{x\in\mathbb{R}:\: \alpha(t)<x<\beta(t), 0\leq t\leq T\}.$$
 The lateral boundary of $%
Q_ {t} $ is given by $\Sigma_{t}=\bigcup_{0\leq t<T} \left(
\left\{\alpha(t),\beta(t)\right\}\times \{t\}\right)$.

In \cite{RACF14}, the authors established the existence, uniqueness and
asymptotic behaviour of strong regular solutions for these type of problems using a
coordinate transformation, which fixes the boundaries, and assuming that the real
function $\gamma(t)=\beta(t)-\alpha(t)$ is increasing on $0\leq t<T$. They used the fact
that, when $(x,t)$ varies in $Q_{t}$, the point $(y,t)$ of $\mathbb{R}^{2}$%
, with $y=(x-\alpha(t))/\gamma(t)$, varies in the cylinder $Q=]0,1[
\times]0,T[$. Thus, the function $\tau :Q_{t} \longrightarrow Q$ given by $%
\tau(x,t)=(y,t)$, is of class $\mathcal{C}^{2}$. The inverse $\tau^{-1}$ is
also of class $\mathcal{C}^{2}$. The change of variable $v(y,t)=u(x,t)$ and $%
g(y,t)=f(x,t)$ with $x=\alpha(t)+\gamma(t)\,y$ transforms problem (\ref%
{prob0}) into the following problem:
\begin{equation}  \label{prob1}
\left\{
\begin{array}{l}
\displaystyle \frac{\partial v_i}{\partial t}-a_i\left(
l(v_1),\dots,l(v_{{n_e}})\right)b_2(t) \frac{\partial^2v_i}{\partial y^2}%
-b_1(y,t)\frac{\partial v_i}{\partial y}=g_i\left(y,t\right)\,, \quad
(y,t)\in Q \\
\displaystyle v_i\left(0,t\right) =v_i\left(1,t\right)=0\,,\quad t>0 \\
\displaystyle v_i(y,0)=v_{i0}(y)\,, \quad y\in \Omega=]0,1[,\quad
i=1,\dots,{n_e} \\
\end{array}
\right. \,
\end{equation}
where $l(v)=\gamma(t)\int_0^1 v(y,t)\ dy$, $g_i(y,t)=f_i(\alpha+\gamma\,y,t)$
and $v_{i0}(y)=u_{i0}(\alpha(0)+\gamma(0)\,y)$. The coefficients $b_{1}(y,t)$
and $b_{2}(t)$ are defined by
\begin{equation*}
b_{1}(y,t)=\frac{\alpha^{\prime }(t)+\gamma^{\prime }(t)y}{\gamma(t)}\quad
\text{and} \quad b_{2}(t)=\frac{1}{\left(\gamma(t)\right)^{2}} \,.
\end{equation*}
With this change of variable, we transfer the problem of the boundary's movement to the first order advection term. If the speed of the boundary grows fast with time, $b_1$ can dominate in magnitude the diffusion coefficient, which can result in numerical instability. Thus, some conditions must be imposed on the mesh size and on the time step. We will address this issue later.\\
Since we need the existence and uniqueness of a strong solution in $Q_{t}$,
we will assume that the hypotheses in \cite{RACF14} are satisfied, namely:
\begin{equation*}
\begin{array}{l}
(H1)\qquad \alpha ,\,\beta \in \mathcal{C}^{2}\left( [0,T]\right) \text{ and
 }0<\gamma _{0}<\gamma (t)<\gamma _{1}<\infty \,,\;\text{for all }t\in
\lbrack 0,T] \\
(H2)\qquad \alpha^{\prime },\,\beta^{\prime}\in L_{2}\left(]0,T[\right) \\
(H3)\qquad u_{i0}\in H_{0}^{1}\left( \Omega _{0}\right) \,,\quad \Omega
_{0}=]\alpha (0),\beta (0)[,\quad i=1,\dots ,{n_e}, \\
(H4) \qquad \int_0^T\int_{\Omega_t} f^2_i\ dxdt< \infty\,, \quad
\Omega_{t}=]\alpha(t),\beta(t)[,\,i=1,\dots ,{n_e}, \\
(H5)\qquad a_{i}:\mathbb{R}^{{n_e}}\longrightarrow \mathbb{R}^{+}\text{
is Lipschitz-continuous } \\
\hspace{1.9cm}\text{with }0<m_{a}\leq a_{i}(s)\leq M_{a}\,,\;\;\text{for all
}s\in \mathbb{R},\quad i=1,\dots ,{n_e}. \\
\end{array}%
\end{equation*}%
We also need to assume that
$$|\gamma'(t)|\leq \gamma'_{\max}\, \text{and}\, |\alpha'(t)|\leq \alpha'_{\max}.$$
Let $\Omega =]0,1[$. The definition of a weak solution is as follows.
\begin{definition}[Weak solution]
\label{fraca} We say that the function $\mathbf{v}=(v_{1},\dots ,v_{{n_e}})$ is
a weak solution of problem (\ref{prob1}) if, for each $i\in \{1,\dots ,{n_e}\}$,
\begin{equation}
v_{i}\in L_{\infty }(0,T;H_{0}^{1}(\Omega )\cap H^{2}(\Omega )),\frac{%
\partial v_{i}}{\partial t}\in L_{2}(0,T;L_{2}(\Omega )),  \label{regx}
\end{equation}
the following equality is valid for all $w_{i}\in H_{0}^{1}(\Omega ),$ and $t\in ]0,T[$,
\begin{equation}
\int_{0}^{1}\frac{\partial v_{i}}{\partial t}w_{i}dy+a_{i}(l(v_{1}),\dots
,l(v_{{n_e}}))b_{2}\int_{0}^{1}\frac{\partial v_{i}}{\partial y}\frac{\partial
w_{i}}{\partial y}dy-\int_{0}^{1}b_1\frac{\partial v_{i}}{\partial y}w_{i}dy=\int_{0}^{1}g_{i}w_{i}dy  \label{fracav}
\end{equation}
 and
\begin{equation}
v_{i}(x,0)=v_{i0}(x),\quad x\in \Omega   \label{condiuv}
\end{equation}
\end{definition}

\section{Semidiscrete solution}
We denote the usual $L_{2}$ norm and inner product in $\Omega $ by $\Vert .\Vert $ and  $(.,.)$ respectively, and the
norm in $H^{k}(\Omega )$ by $\Vert .\Vert _{H^{k}}$.
Let $\mathcal{T}_{h}$ denote a partition of $\Omega $ into disjoint
intervals $T_{i}$, $i=1,\dots ,n_t$, such that $h=\max \{diam(T_{i}),i=1,\dots
,n_t\}$. Now, let $S_{h}^{k}$ denote the continuous functions on the closure $%
\bar{\Omega}$ of $\Omega $ which are polynomials of degree $k$ in each
interval of $\mathcal{T}_{h}$ and which vanish on $\partial \Omega $, that
is,
\begin{equation*}
S_{h}^{k}=\{W\in C_{0}^{0}(\bar{\Omega})|{W}_{|{T_{i}}}\text{ is a
polynomial of degree $k$ for all }T_{i}\in \mathcal{T}_{h}\}.
\end{equation*}%
If $\{\varphi _{j}\}_{j=1}^{n_p}$ is the Lagrange basis for $S_{h}^{k}$ associated to the points $\{P_{j}\}_{j=1}^{n_p}$, then we can
represent each $W\in S_{h}^{k}$ as
\begin{equation*}
W=\sum_{j=1}^{n_p}W(P_{j})\varphi _{j}.
\end{equation*}%
Given a smooth function $u$ on $\Omega,$ which vanishes on $\partial \Omega
$, we may define its interpolant, denoted by $I_{h}u$, as the function of $%
S_{h}^{k}$ which coincides with $u$ at the points $\{P_{j}\}_{j=1}^{n_p}$,
that is,
\begin{equation*}
I_{h}u=\sum_{j=1}^{n_p}u(P_{j})\varphi _{j}.
\end{equation*}

\begin{lemma}[\protect\cite{Tho06}]
\label{errint} If $u\in H^{k+1}(\Omega)\cap H_0^1(\Omega)$, then
\begin{equation*}
\|I_h u-u\|+h\|\nabla(I_h u-u)\|\leq Ch^{k+1}\|u\|_{H^{k+1}}.
\end{equation*}
\end{lemma}

\begin{definition}[\protect\cite{Tho06}]
A function $\tilde u\in S_h^k$ is said to be the Ritz projection of $u\in
H_0^1(\Omega)$ onto $S_h^k$ if it satisfies
\begin{equation*}
\es{\nabla \tilde u}{\nabla W}=\es{\nabla
u}{\nabla W},\quad \text{ for all } W\in S_h^k.
\end{equation*}
\end{definition}

\begin{lemma}[\protect\cite{Tho06}]
\label{errproj} If $u\in H^{k+1}(\Omega )\cap H_{0}^{1}(\Omega )$, then
\begin{equation*}
\Vert \tilde{u}-u\Vert +h\Vert \nabla (\tilde{u}-u)\Vert \leq Ch^{k+1}\Vert
u\Vert _{H^{k+1}},
\end{equation*}%
where $C$ does not depend on $h$ nor on $k$.
\end{lemma}

The semidiscrete problem, based on Definition \ref{fraca}, consists in
finding $\mathbf{V}=(V_1,\dots,V_{{n_e}})\in(S_h^k)^{{n_e}}$, for $
t\geq 0$, such that for all $\mathbf{W}=(W_1,\dots,W_{{n_e}})\in (S_h^k)^{{n_e}}$
and $t\in]0,T[$:
\begin{equation}  \label{probsd}
\left\{%
\begin{array}{l}
\displaystyle \es{\frac{\partial V_{i}}{\partial t}}{
W_i}+a_i(l(V_1),\dots,l(V_{{n_e}}))b_{2}\es{\frac{\partial V_i}{
\partial y}}{\frac{\partial W_i}{\partial y}}-\es{b_{1}\frac{
\partial V_i}{\partial y}}{W_i}=\es{g_i}{W_i}\\
V_i(y,0)=I_hv_{i0},\quad i=1,\dots,{n_e} \\
\end{array}%
\right..
\end{equation}
Since the functions $a_i$ are continuous, Caratheodory´s Theorem implies the existence of a solution to system (\ref{probsd}), and arguing as in the proof of Theorem 3 in \cite{RACF14}, we can prove the uniqueness of this solution.
In virtue of condition (H5), the convergence of the semidiscrete solution to the weak solution of problem (\ref{prob1}) can be obtained using standard arguments, and hence we will only present the main steps of the proof and specify the dependence on the regularity of the weak solution.

\begin{theorem}
\label{conv_h} If $\mathbf{v}$ is the solution of problem (\ref{prob1}) and $\mathbf{V}$ is the solution of problem (\ref{probsd}), then
\begin{equation*}
\Vert V_{i}-v_{i}\Vert \leq Ch^{k+1},\quad t\in ]0,T],\quad i=1,\dots ,{n_e}
\end{equation*}%
where $C$ may depend on $\sum_{i=1}^{n_e}\left\Vert
v_{i0}\right\Vert_{H^{k+1}}$, $\sum_{i=1}^{n_e}\|v_i\|_{L_{\infty}(0,T;H^{k+1}(\Omega))}$,\\ $\sum_{i=1}^{n_e}\left\Vert \pd {v_i}{y}\right\Vert_{L_{\infty}(0,T;L_2(\Omega))}$, $\sum_{i=1}^{n_e}\left\Vert v_i\right\Vert
_{L_{2}(0,T;H^{k+1}(\Omega))}$and $\sum_{i=1}^{n_e}\|\pd {v_i}{t}\|_{L_2(0,T;H^{k+1}(\Omega))}$,\\  but does not depend on $h$, $k$ or $i$.
\end{theorem}

\begin{proof}
Let $e_i=V_i-v_i$ be written as
\begin{equation*}
e_i(y,t)=(V_i(y,t)-\tilde{V}_i(y,t))+(\tilde{V}_i
(y,t)-v_i(y,t))=\theta_i (y,t)+\rho_i (y,t),
\end{equation*}
with  $\tilde{V}_i(y,t)\in S_{h}^{k}$ being the Ritz projection of $v_i$. Then
\begin{equation*}
\left\Vert e_i(y,t)\right\Vert\leq \left\Vert \theta_i
(y,t)\right\Vert+\left\Vert \rho_i (y,t)\right\Vert
\end{equation*}
and, by Lemma \ref{errproj}, it follows that
\begin{equation*}
\left\Vert \rho_i (y,t)\right\Vert\leq Ch^{k+1}\left\Vert
v_i\right\Vert _{H^{k+1}},\quad t\in[0,T].
\end{equation*}
Concerning $\left\Vert \theta_i (y,t)\right\Vert$, if
\begin{equation*}
a_{i}^{(h)}=a_i(l(V_1),\dots,l(V_{{n_e}})),
\end{equation*}
then, for every $i\in\{1,\dots,{n_e}\}$, we have that\\

$\displaystyle\es{\pd{\theta_i}{t}}{ W_i}+a_{i}^{(h)}b_{2}\es{\frac{\partial \theta_i
}{\partial y}}{\frac{\partial W_i}{\partial y}}-\es{b_{1}\frac{
\partial \theta_i }{\partial y}}{W_i} $
\begin{eqnarray*}
&=&\es{\pd{V_i}{t}}{W_i}+a_{i}^{(h)}b_{2}\es{\frac{\partial V_i
}{\partial y}}{\frac{\partial W_i}{\partial y}}-\es{b_{1}\frac{
\partial V_i}{\partial y}}{W_i} \\
&&-\es{\pd{\tilde{V}_{i}}{t}}{W_i}-a_{i}^{(h)}b_{2}\es{
\frac{\partial \tilde{V}_i}{\partial y}}{\frac{\partial W_i}{
\partial y}}+\es{b_{1}\frac{\partial \tilde{V}_i}{
\partial y}}{W_i} \\
&=&\es{g_i}{W_{i}}-\es{\pd{v_{i}}{t}}{W_i}-a_ib_{2}\es{\frac{
\partial \tilde{V}_i}{\partial y}}{\frac{\partial W_i}{
\partial y}}+\es{b_{1}\frac{\partial v_i}{\partial y}}{W_i} \\
&&+(a_i-a_{i}^{(h)})b_{2}\es{\frac{\partial \tilde{V}_i}{
\partial y}}{\frac{\partial W_i}{\partial y}}+\es{b_{1}(\frac{\partial \tilde{V}_i}{\partial y}-
\frac{\partial v_i}{\partial y})}{W_i}\\
&&+\es{\pd{v_{i}}{t}-\pd{\tilde{V}_{i}}{t}}{W_i} \\
&=&(a_i-a_{i}^{(h)})b_{2}\es{\frac{\partial \tilde{V}_i}{
\partial y}}{\frac{\partial W_i}{\partial y}}+\es{b_{1}(\frac{
\partial \tilde{V}_i}{\partial y}-\frac{\partial v_i}{\partial y}
)}{W_i}\\
&&+\es{\pd{v_{i}}{t}-\pd{\tilde{V}_{i}}{t}}{W_i}.
\end{eqnarray*}
If we consider $W_i=\theta_i$, then
$$\es{\pd{\theta_{i}}{t}}{\theta_i}+a_{i}^{(h)}b_{2}\left\| \frac{\partial \theta_i }{\partial y}\right\| ^{2}
=(a_i-a_{i}^{(h)})b_{2}\es{\frac{\partial \tilde{V}_i}{\partial y}}{\frac{\partial \theta_i }
{\partial y}}+\es{b_{1}\frac{\partial \rho_i }{\partial y}}{\theta_i}-\es{\pd{\rho_{i}}{t}}{\theta_i }$$$$+\es{b_{1}\frac{\partial \theta_i }{\partial y}}{\theta_i}.$$
Integrating by parts the second and the fourth terms on the right side of the above equation, we obtain\\

$\displaystyle
\frac12\frac{d}{dt}\|\theta_i\|^2+a_{i}^{(h)}b_{2}\left\| \frac{\partial \theta_i }{\partial y}\right\|^{2}$\\

\begin{flushright}
$\displaystyle
=(a_i-a_{i}^{(h)})b_{2}\es{\frac{\partial \tilde{V}_i
}{\partial y}}{\frac{\partial \theta_i }{\partial y}}-\es{\pd{\rho_{i}}{t}}{\theta_i}-\frac{\gamma ^{^{\prime }}(t)}{\gamma (t)}
\es{\rho_i}{ \theta_i}-\es{b_{1}\rho_i}{ \frac{\partial \theta_i }{
\partial y}}-\frac{\gamma ^{^{\prime }}(t)}{2\gamma (t)}\es{\theta_i}{\theta_i}.$
\end{flushright}
Taking the absolute value of the expression on the right hand side of this equation and considering the lower limits of $a_i$ and $b_{i}$, it follows that\\
$\displaystyle\frac12\frac{d}{dt}\|\theta_i\|^2+\frac{m_a}{\gamma_1^2}\left\|\pd{\theta_i}{y}\right\|^2$\\

$\displaystyle
\leq\left\vert a_i-a_{i}^{(h)}\right\vert \frac{1}{\gamma _{0}^{2}}\int_{0}^{1}\left\vert \frac{\partial\tilde{V}_i}
{\partial y}\right\vert \left\vert \frac{\partial \theta_i }{\partial y}\right\vert dy+\int_{0}^{1}\left\vert \pd{\rho_{i}}{t}\right\vert
\left\vert \theta_i \right\vert dy+\frac{\gamma _{\max }^{^{\prime }}}{\gamma _{0}}\int_{0}^{1}\left\vert
\rho_i \right\vert \left\vert \theta_i \right\vert dy$\\

$\displaystyle\quad +\frac{\alpha _{\max
}^{\prime }+\gamma _{\max }^{^{\prime }}}{\gamma _{0}}\int_{0}^{1}\left\vert
\rho_i \right\vert \left\vert \frac{\partial \theta_i }{\partial y}\right\vert dy+\frac{\gamma_{\max} ^{^{\prime }}}{2\gamma_0}\int_{0}^{1}|\theta_i|
^{2}dy$\\

$\displaystyle\leq C_{1}\left\vert a_i-a_{i}^{(h)}\right\vert^2+ \frac{m_a}{2\gamma _{1}^{2}}
\int_{0}^{1}\left\vert \frac{\partial \theta_i }{\partial y}\right\vert ^{2}dy+
\frac{1}{2}\int_{0}^{1}\left\vert \pd{\rho_{i}}{t}\right\vert ^{2}dy+\frac{1}{2}
\int_{0}^{1}\left\vert \theta_i \right\vert ^{2}dy$\\

$\displaystyle\quad +\frac{\gamma _{\max }^{^{\prime }}}{2\gamma _{0}}\int_{0}^{1}\left\vert
\rho_i \right\vert ^{2}dy+\frac{\gamma _{\max }^{^{\prime }}}{2\gamma _{0}}
\int_{0}^{1}\left\vert \theta_i \right\vert ^{2}dy+C_{2}\int_{0}^{1}\left\vert
\rho_i \right\vert ^{2}dy+\frac{m_{a}}{2\gamma _{1}^{2}}\int_{0}^{1}\left\vert
\frac{\partial \theta_i }{\partial y}\right\vert ^{2}dy,$\\

with $C_1=C_1(m_a,\gamma_0,\left\Vert \pd {\tilde V_i}{y}\right\Vert_{L_{\infty}(0,T;L_2(\Omega))})$.\\ Since $\left\Vert \pd {\tilde V_i}{y}\right\Vert\leq \left\Vert \pd {v_i}{y}\right\Vert$, we have that $C_1=C_1(m_a,\gamma_0,\left\Vert \pd {v_i}{y}\right\Vert_{L_{\infty}(0,T;L_2(\Omega))})$.
Then, by (H5),
\begin{eqnarray*}
\frac{1}{2}\frac{d}{dt}\left\Vert \theta_i \right\Vert ^{2} &\leq
&C_{3}\sum_{j=1}^{{n_e}}\left\Vert \rho_j \right\Vert^{2}+C_{4}\sum_{j=1}^{{n_e}}\left\Vert \theta_j
\right\Vert ^{2}+\frac{1}{2}\left\Vert \pd{\rho_{i}}{t}\right\Vert
^{2}+\frac{1}{2}\left\Vert \theta_i \right\Vert^{2}
\\
&&+\frac{\gamma _{\max }^{^{\prime }}}{2\gamma _{0}}\left\Vert \rho_i
\right\Vert^{2}+\frac{\gamma _{\max }^{^{\prime }}}{2\gamma
_{0}}\left\Vert \theta_i \right\Vert^{2}+C_{2}\left\Vert \rho_i
\right\Vert^{2} \\
&\leq &C\sum_{j=1}^{{n_e}}\left\Vert \theta_j \right\Vert^{2}+C\sum_{j=1}^{{n_e}}\left\Vert \rho_j
\right\Vert^{2}+\frac{1}{2}\left\Vert \pd{\rho_{i}}{t}\right\Vert^{2}.
\end{eqnarray*}
and now $C=C(m_a,\gamma_0,\gamma _{\max }^{^{\prime }},\left\Vert \pd {v_i}{y}\right\Vert_{L_{\infty}(0,T;L_2(\Omega))})$.
Hence, we obtain\\
$$ \frac{d}{dt}\left(\sum_{i=1}^{{n_e}}\left\Vert \theta_i \right\Vert ^{2}\right)\leq C\sum_{i=1}^{{n_e}}\left\Vert \theta_i \right\Vert^{2}+C\sum_{i=1}^{{n_e}}\left\Vert \rho_i
\right\Vert^{2}+\sum_{i=1}^{{n_e}}\left\Vert \pd{\rho_{i}}{t}\right\Vert^{2}.$$
Applying Gronwall's Theorem, we arrive at the inequality
\begin{equation*}
\sum_{i=1}^{{n_e}}\left\Vert \theta_i \right\Vert^{2}\leq C\sum_{i=1}^{{n_e}}\left\Vert \theta_i
(y,0)\right\Vert^{2}+C\sum_{i=1}^{{n_e}}\int_{0}^{T}\left\Vert \rho_i \right\Vert^{2}+\left\Vert \pd{\rho_{i}}{t}\right\Vert^{2}dt.
\end{equation*}
By the hypothesis of the theorem, we have, for every $i\in\{1,\dots,{n_e}\}$,
\begin{eqnarray*}
\left\Vert \theta_i (y,0)\right\Vert^{2} &\leq &\left\Vert
e_i(y,0)\right\Vert^{2}=\left\Vert V_i(y,0)-v_{i0}\right\Vert
^{2}\leq Ch^{2(k+1)}\|v_{i0}\|_{H^{k+1}}^{2}, \\
\int_{0}^{T}\left\Vert \rho_i \right\Vert^{2} \ dt&\leq
&CTh^{2(k+1)}\left\Vert v_i\right\Vert _{L_2(0,T;H^{k+1}(\Omega))}^{2},\\
\int_{0}^{T}\left\Vert \pd{\rho_{i}}{t} \right\Vert^{2}\ dt &\leq
&CTh^{2(k+1)}\left\|\pd{v_{i}}{t}\right\|_{L_2(0,T;H^{k+1}(\Omega))}^2
\end{eqnarray*}
and so
$$\displaystyle\sum_{i=1}^{{n_e}}\left\Vert \theta_i \right\Vert^{2}\leq C\left( \sum_{i=1}^{{n_e}}\left\Vert
v_{i0}\right\Vert _{H^{k+1}}^{2}+\sum_{i=1}^{{n_e}}\left\Vert v_i\right\Vert
_{L_2(0,T;H^{k+1}(\Omega))}^{2} \right.$$
$$\left.+\sum_{i=1}^{{n_e}}\left\Vert \pd{v_i}{t}\right\Vert
_{L_2(0,T;H^{k+1}(\Omega))}^{2} \right) h^{2(k+1)}.
$$
Hence
\begin{equation*}
\left\Vert \theta_i \right\Vert\leq Ch^{k+1},\quad i=1,\dots {n_e}
\end{equation*}
and adding the estimate of $\rho_i$, we obtain the desired result.
\end{proof}
It is important to note that Gronwall's constant depends on the ratio $\frac{\gamma'_{\max}}{\gamma_0}$. So, if $\gamma'_{\max}$ is high and $\gamma_0$ is small, then, for long time computations, the mesh size should be small enough to compensate for this behaviour.

\section{Discrete problem}
In this section, we will study the applicability of three known finite diference schemes to discretise in time equation (\ref{probsd}). At the end, we will comment the results.
Let $\delta >0$ and consider the partition $]0,T]=\overset{n_{i}-1}{\underset%
{j=1}{\cup }}]t_{j-1},t_{j}]=\overset{n_{i}-1}{\underset{j=1}{\cup }}I_{j},$
$\delta =t_{j}-t_{j-1}$ and $int(I_{j})\cap int(I_{i})=\emptyset $. Let $\mathbf{V}^{(n)}(y)$ be the approximation of $\mathbf{v}(y,t_n)$ in $(S_h^k)^{n_e}$. In the subsequente, the notation $V^{(n)}$ represents the function $V$ evaluated at time $t_n$.


\subsection{Backward Euler method}

First we are going to study the backward Euler method.
This method  evaluates the  equation at the points $t_{n+1}$, $n=0,\dots, n_i-1$, and approximates the time derivative by
$$\pd{\mathbf{V}}{t}(y,t_{n+1})\approx \frac{\mathbf{V}^{(n+1)}(y)-\mathbf{V}^{(n)}(y)}{\delta}=\bar\partial \mathbf{V}^{(n+1)}(y).$$
In this case,  system (\ref{probsd}) becomes
$$\es{\bar\partial V_i^{(n+1)}(y)}{W_i}+b_2^{(n+1)}a_i^{(n+1)}\es{\pd{V_i^{(n+1)}}{y}}{\pd{W_i}{y}}-\es{b_1^{(n+1)}\pd{V_i^{(n+1)}}{y}}{W_i} $$
\begin{equation}\label{eq_E_im}
=\es{g_i^{(n+1)}}{W_i},
\end{equation}
with $a_i^{(n+1)}=a_i(l^{(n+1)}(V_1^{(n+1)}),\dots,l^{(n+1)}(V_{n_e}^{(n+1)}))$.
Recalling the basis $\{\varphi_j\}_{j=1}^{n_p}$, system (\ref{eq_E_im}) is a nonlinear algebraic system of the form
$$(\mathcal{M}+\delta \mathcal{A}a(\mathcal{V}^{(n+1)})-\delta \mathcal{B})\mathcal{V}^{(n+1)}=\mathcal{M}\mathcal{V}^{(n)}+\delta \mathcal{G},$$
with the unknown $$\mathcal{V}^{(n+1)}=(V_{1,1}^{(n+1)},\dots,V_{1,n_p}^{(n+1)},\dots,V_{n_e,1}^{(n+1)},\dots,\dots,V_{n_e,n_p}^{(n+1)}).$$
Due to its nonlinearity, we need to prove the existence of a solution.
\begin{theorem}\label{exi_E_im}
For each  $n=0,\dots, n_i-1$, system (\ref{eq_E_im}) has a solution.
\end{theorem}

\begin{proof}
Let $n\geq0$ be fixed. For each $h,\delta>0$, we define the continuous mapping $F:S_{h}^{k}\to S_{h}^{k}$ by
$$\es{F(V_i^{(n+1)})}{W_i} =\es{V_i^{(n+1)}}{W_i} -\es{ V_i^{(n)}}{W_i} +\delta b_2^{(n+1)}a_i^{(n+1)}\es{\pd{V_i^{(n+1)}}{y}}{\pd{W_i}{y}}$$
$$-\delta\es{b_1^{(n+1)}\pd{V_i^{(n+1)}}{y}}{W_i}  -\delta\es{g_i^{(n+1)}}{W_i}.$$
If $W_i=V_i^{(n+1)}$, then
$$\es{F(V_i^{(n+1)})}{V_i^{(n+1)}}=\es{V_i^{(n+1)}}{V_i^{(n+1)}} -\es{V_i^{(n)}}{V_i^{(n+1)}} -\delta\es{g_i^{(n+1)}}{V_i^{(n+1)}}$$
$$-\delta\es{b_1^{(n+1)}\pd{V_i^{(n+1)}}{y}}{V_i^{(n+1)}} +\delta b_2^{(n+1)}a_i^{(n+1)}\es{\pd{V_i^{(n+1)}}{y}}{\pd{V_i^{(n+1)}}{y}}$$
$$\geq \left\|V_i^{(n+1)}\right\|^2+\delta\frac{\gamma'_{\max}}{2\gamma_0}\left\|V_i^{(n+1)}\right\|^2
-\delta\|g_i^{(n+1)}\|\|V_i^{(n+1)}\|+\delta\frac{CM_a}{\gamma_0^2}\left\|V_i^{(n+1)}\right\|^2$$
$$-\|V_i^{(n)}\|\|V_i^{(n+1)}\|$$
$$=\left\|V_i^{(n+1)}\right\|\left(\left\|V_i^{(n+1)}\right\|+\delta\frac{\gamma'_{\max}}{2\gamma_0}\left\|V_i^{(n+1)}\right\|
-\delta\|g_i^{(n+1)}\|+\delta\frac{CM_a}{\gamma_0^2}\left\|V_i^{(n+1)}\right\|\right.$$ $$\left.-\|V_i^{(n)}\|\right)$$
$$=\left\|V_i^{(n+1)}\right\|\left(\left(1+\delta\frac{\gamma'_{\max}}{2\gamma_0}+\delta\frac{CM_a}{\gamma_0^2}\right)
\left\|V_i^{(n+1)}\right\|-\delta\|g_i^{(n+1)}\|-\|V_i^{(n)}\|\right)$$
Let us define $$\varepsilon>\frac{\delta\|g_i^{(n+1)}\|+\|V_i^{(n)}\|}{1+\delta\frac{\gamma'_{\max}}{2\gamma_0}+\delta\frac{CM_a}{\gamma_0^2}}$$
and
$$B_{\varepsilon}=\{W\in S_{h}^{k}: \|W\|\leq \varepsilon\}.$$
Since $(F(V),V)>0$ for every $V\in\partial B_{\varepsilon}$ the corollary to the Brower's Fixed Point Theorem implies the existence of a solution to problem (\ref{eq_E_im}).
\end{proof}
The stability of this method is proved under a condition on the time step.
\begin{theorem} \label{est_E_im}
Let $\mathbf{V}^{(n+1)}(y)$ be the solution of equation (\ref{eq_E_im}). If
\begin{equation}\label{cond_d_1}
\delta<\frac{\gamma_0}{\gamma_0+\gamma'_{\max}},
\end{equation}
 then
$$\left\|V_i^{(n+1)}\right\|^2\leq C^{n+1}\left\|V_i^{(0)}\right\|^2+\sum_{l=0}^{n+1}C^{n-l+2}\delta\left\|g_i^{(l)}\right\|^2,$$
where $C$ could depend on  $\gamma'_{\max}$, $\gamma_0$ and $\delta$.
\end{theorem}

\begin{proof}
Setting $W_i=V_i^{n+1}$ in (\ref{eq_E_im}), we obtain the equality
$$\es{V_i^{(n+1)}}{V_i^{(n+1)}} +\delta b_2^{(n+1)}a_i^{(n+1)}\es{\pd{V_i^{(n+1)}}{y}}{\pd{V_i^{(n+1)}}{y}} =\es{V_i^{(n)}}{V_i^{(n+1)}}$$
$$ -\frac{\delta(\gamma')^{(n+1)}}{2(\gamma)^{(n+1)}}\es{V_i^{(n+1)}}{V_i^{(n+1)}} +\delta\es{g_i^{(n+1)}}{V_i^{(n+1)}}.$$
Thus
$$\left\|V_i^{(n+1)}\right\|^2\leq\frac12 \left\|V_i^{(n+1)}\right\|^2+\frac12 \left\|V_i^{(n)}\right\|^2 +\frac{\delta\gamma'_{\max}}{2\gamma_0}\left\|V_i^{(n+1)}\right\|^2$$
$$+\frac{\delta}{2}\left\|g_i^{(n+1)}\right\|^2+\frac{\delta}{2}\left\|V_i^{(n+1)}\right\|^2.$$
$$\left(\frac12-\frac{\delta\gamma'_{\max}}{2\gamma_0}-\frac{\delta}{2}\right)\left\|V_i^{(n+1)}\right\|^2\leq \frac12\left\|V_i^{(n)}\right\|^2+\frac{\delta}{2}\left\|g_i^{(n+1)}\right\|^2$$
From (\ref{cond_d_1}), it now follows that
$$\left\|V_i^{(n+1)}\right\|^2\leq C\left\|V_i^{(n)}\right\|^2 +\delta C\left\|g_i^{(n+1)}\right\|^2,$$
with $C=C(\gamma_0,\gamma'_{\max},\delta)$.
Iterating the result follows.
\end{proof}
As we suspected, the stability of this method depends on $\delta$ and it could be affected if $\delta$  is not sufficiently small to compensate for the ratio $\frac{\gamma'_{\max}}{\gamma_0}$.\\
The uniqueness of the solution is proved in the next theorem.
\begin{theorem} \label{uni_E_im}
If $\delta\approx h^2$ is sufficiently small, then the solution of equation (\ref{eq_E_im}) is unique.
\end{theorem}

\begin{proof}
Suppose that equation (\ref{eq_E_im}) has two distinct solutions $\mathbf{X}$ and $\mathbf{Y}$, then
$$(\bar\partial X_i,W_i)+b_2^{(n+1)}a_{i}(l^{(n+1)}(X_1),\dots,l^{(n+1)}(X_{n_e}))
\left(\pd{X_i}{y},\pd{W_i}{y}\right)$$ $$-\left(b_1^{(n+1)}\pd{X_i}{y},W_i\right)=(g_i^{(n+1)},W_i)$$
and
$$(\bar\partial Y_i,W_i)+b_2^{(n+1)}a_{i}(l^{(n+1)}(Y_1),\dots,l^{(n+1)}(Y_{n_e}))
\left(\pd{Y_i}{y},\pd{W_i}{y}\right)$$ $$-\left(b_1^{(n+1)}\pd{Y_i}{y},W_i\right)=(g_i^{(n+1)},W_i)$$
Subtracting, we arrive at
$$(Y_i-Y_i,W_i)+\delta b_2^{(n+1)}\left(a_{i,1}^{(n+1)}\pd{X_i}{y}-a_{i,2}^{(n+1)}\pd{Y_i}{y},\pd{W_i}{y}\right)$$ $$-\delta\left(b_1^{(n+1)}\left(\pd{X_i}{y}-\pd{Y_i}{y}\right),W_i\right)=0.$$
Defining $E_i=Y_i-Y_i$, it follows that
$$(E_i,W_i)+\delta b_2^{(n+1)}a_{i,2}^{(n+1)}\left(\pd{E_i}{y},\pd{W_i}{y}\right)=\delta\left(b_1^{(n+1)}\pd{E_i}{y},W_i\right)+
\delta(a_{i,2}^{(n+1)}$$
$$-a_{i,1}^{(n+1)})\left(\pd{X_i}{y},\pd{W_i}{y}\right).$$
Setting $W_i=E_i$, we obtain
$$\|E_i\|^2+\frac{\delta m_a}{\gamma_1^2}\left\|\pd{E_i}{y}\right\|^2\leq \frac{\delta \gamma'_{\max}}{2\gamma_0}\|E_i\|^2+\frac{\delta\gamma_1^2 C}{4m_a}\left\|\pd{X_i}{y}\right\|_{\infty}\sum_{j=1}^{n_e}\|E_j\|^2+\frac{\delta m_a}{\gamma_1^2}\left\|\pd{E_i}{y}\right\|^2,$$
whence
$$\sum_{j=1}^{n_e}\|E_j\|^2\leq \frac{\delta \gamma'_{\max}}{2\gamma_0}\sum_{j=1}^{n_e}\|E_j\|^2+ \frac{\delta\gamma_1^2 C}{4m_a}\left\|\pd{X_i}{y}\right\|_{\infty}\sum_{j=1}^{n_e}\|E_j\|^2$$
and thus
$$\left(1-\frac{\delta \gamma'_{\max}}{2\gamma_0}-  \frac{\delta\gamma_1^2 C}{4m_a}\left\|\pd{X_i}{y}\right\|_{\infty}\right)\sum_{j=1}^{n_e}\|E_j\|^2\leq 0.$$
Using the inverse estimates valid in $S_h^k$, we can prove that
$$\left\|\pd{X_i}{y}\right\|_{\infty}\leq Ch^{-1}\left\|\pd{X_i}{y}\right\|\leq h^{-2}\left\|X_i\right\|.$$
By Theorem \ref{est_E_im}, the result is proved, provided that  $\delta\approx h^2$ is sufficiently small.
\end{proof}

The next theorem establishes optimal convergence order conditions for this scheme.

\begin{theorem}  \label{conv_E_im}
Suppose that $\delta$ is small. If $\mathbf{v}$ is the solution of (\ref{prob1}) and
$\mathbf{V}^{(n+1)}$ is the solution of (\ref{eq_E_im}), then
\begin{equation*}
\Vert V_{i}^{(n+1)}(y)-v_{i}(y,t_{n+1})\Vert \leq C(h^{k+1}+\delta),\quad i=1,\dots ,{n_e},\quad
n=1,\dots ,n_i,
\end{equation*}
where $C$ does not depend on $h$, $k$ or $\delta $, but could depend on $\gamma_0$, $\gamma_1$, $m_a$, $\gamma_{\max}'$, $\alpha'_{\max}$, $\left\|\pd{{}^2\mathbf{v}}{t^2}\right\|_{L_\infty(0,T;L_2(0,1))}$,  $\left\|\pd{\mathbf{v}}{t}\right\|_{L_\infty(0,T;H^{k+1}(0,1))}$, $\left\|\pd{\mathbf{v}}{y}\right\|_{L_\infty(0,T;L_2(0,1))}$, $\left\|\pd{\mathbf{v}}{y}\right\|_{L_\infty(0,T;H^{k+1}(0,1))}$  and $\left\|\mathbf{v}\right\|_{L_\infty(0,T;H^{k+1}(0,1))}$.
\end{theorem}
\begin{proof}
Set $V_i^{(n+1)}-v_i^{(n+1)}=V_i^{(n+1)}-\tilde v_i^{(n+1)}+\tilde v_i^{(n+1)}=\theta_i^{(n+1)}+\rho_i^{(n+1)}.$ By Lemma \ref{errproj}, we have that
$$\left\Vert \rho_i^{(n+1)}\right\Vert\leq Ch^{k+1}\left\Vert
v_i\right\Vert _{H^{k+1}},\quad n=1,\dots, n_i.$$
For $\theta_i$, we set
$$\es{\bar\partial\theta_i^{(n+1)}}{W_i}+b_2^{(n+1)}a_i^{(n+1)}\es{\pd{\theta_i^{(n+1)}}{y}}{\pd{W_i}{y}} -\es{b_1^{(n+1)}\pd{\theta_i^{(n+1)}}{y}}{W_i}$$
$$=\es{\bar\partial V_i^{(n+1)}}{W_i} +b_2^{(n+1)}a_i^{(n+1)}\es{\pd{V_i^{(n+1)}}{y}}{\pd{W_i}{y}} -\es{b_1^{(n+1)}\pd{V_i^{(n+1)}}{y}}{W_1}$$
$$-\es{\bar\partial\tilde v_i^{(n+1)}}{W_i} -b_2^{(n+1)}a_i^{(n+1)}\es{\pd{\tilde v_i^{(n+1)}}{y}}{\pd{W_i}{y}} +\es{b_1^{(n+1)}\pd{\tilde v_i^{(n+1)}}{y}}{W_i}$$
$$=\es{g_i^{(n+1)}}{W_i}-
\es{\bar\partial\tilde v_i^{(n+1)}}{W_i}-b_2^{(n+1)}a_i^{(n+1)}\es{\pd{\tilde v_i^{(n+1)}}{y}}{\pd{W_i}{y}}$$ $$+\es{b_1^{(n+1)}\pd{\tilde v_i^{(n+1)}}{y}}{W_i}$$
$$=\es{\left(\pd{v_i}{t}\right)^{(n+1)}}{W_i} +b_2^{(n+1)}a_i^{(n+1)}(\mathbf{v}^{(n+1)})\es{\pd{v_i^{(n+1)}}{y}}{\pd{W_i}{y}}$$ $$-\es{b_1^{(n+1)}\pd{v_i^{(n+1)}}{y}}{W_i}-\es{\bar\partial\tilde v_i^{(n+1)}}{W_i}-b_2^{(n+1)}a_i^{(n+1)}(\mathbf{V}^{(n+1)})\es{\pd{v_i^{(n+1)}}{y}}{\pd{W_i}{y}}$$
$$ +\es{b_1^{(n+1)}\pd{\tilde v_i^{(n+1)}}{y}}{W_i}$$
$$=\es{\left(\pd{v_i}{t}\right)^{(n+1)}-\bar\partial\tilde v_i^{(n+1)}}{W_i} +b_2^{(n+1)}(a_i^{(n+1)}(\mathbf{v}^{(n+1)})$$
$$-a_i^{(n+1)}(\mathbf{V}^{(n+1)}))\es{\pd{ v_i^{(n+1)}}{y}}{\pd{W_i}{y}} +\es{b_1^{(n+1)}\left(\pd{\tilde v_i^{(n+1)}}{y}-\pd{v_i^{(n+1)}}{y}\right)}{W_i}.$$
Making $W_i=\theta_i^{(n+1)}$ and taking in to account the lower bounds of $a$ and $b_2$, we obtain
$$\frac12\bar\partial\ \|\theta_i^{(n+1)}\|^2+\frac{m_a}{\gamma_1^2}\left\|\pd{\theta_i^{(n+1)}}{y}\right\|^2\leq \es{b_1^{(n+1)}\pd{\theta_i^{(n+1)}}{y}}{\theta_i^{(n+1)}}$$
$$+\es{\left(\pd{v_i}{t}\right)^{(n+1)}-\bar\partial\tilde v_i^{(n+1)}}{\theta_i^{(n+1)}} +\es{b_1^{(n+1)}\left(\pd{\tilde v_i^{(n+1)}}{y}-\pd{v_i^{(n+1)}}{y}\right)}{\theta_i^{(n+1)}} $$
\begin{equation}\label{eq_aux_E_im}
+b_2^{(n+1)}(a_i^{(n+1)}(\mathbf{v}^{(n+1)})-a_i^{(n+1)}(\mathbf{V}^{(n+1)}))\es{\pd{v_i^{(n+1)}}{y}}{\pd{\theta_i^{(n+1)}}{y}}.
\end{equation}
Using the hypothesis $(H1)$ and $(H2)$ and integration by parts, we obtain
$$\es{b_1^{(n+1)}\pd{\theta_i^{(n+1)}}{y}}{\theta_i^{(n+1)}}=-\frac{(\gamma')^{(n+1)}}{2\gamma^{(n+1)}}\left\|\theta_i^{(n+1)}\right\|^2.$$
By $(H5)$, we have
$$|a_i^{(n+1)}(\mathbf{v}^{(n+1)})-a_i^{(n+1)}(\mathbf{V}^{(n+1)})|\leq |\gamma^{(n+1)}|\sum_{i=1}^{n_e}C_i\|v_i^{(n+1)}-V_i^{(n+1)}\|$$
$$\leq \gamma_1\left(\sum_{i=1}^{n_e}C_i\|\theta_i^{(n+1)}\|+\sum_{i=1}^{n_e}C_i\|\rho_i^{(n+1)}\|\right).$$
Taking the absolute value of the expression on the right-hand side of inequality (\ref{eq_aux_E_im}) and using the Cauchy inequality, it follows that
$$\frac12\bar\partial\ \|\theta_i^{(n+1)}\|^2+\frac{m_a}{\gamma_1^2}\left\|\pd{\theta_i^{(n+1)}}{y}\right\|^2\leq\frac{\gamma'_{\max}}{2\gamma_0}\|\theta_i^{(n+1)}\|^2 +\frac12\left\|\left(\pd{v_i}{t}\right)^{(n+1)}-\bar\partial\tilde v_i^{(n+1)}\right\|^2$$
$$+\frac12\|\theta_i^{(n+1)}\|^2 +\left\|\pd{v_i}{y}\right\|_{L_\infty(0,T;L_2(0,1))}\frac{\gamma_1^2}{4m_a}\gamma_1^2\left(\sum_{i=1}^{n_e}C_i\|\theta_i^{(n+1)}\|^2+
\sum_{i=1}^{n_e}C_i\|\rho_i^{(n+1)}\|^2\right)$$
$$+\frac{m_a}{\gamma_1^2}\left\|\pd{\theta_i^{(n+1)}}{y}\right\|^2 +\frac{(\alpha'_{\max}+\gamma_{\max}')^2}{2\gamma_0^2}\left\|\pd{\rho_i^{(n+1)}}{y}\right\|^2 +\frac12\|\theta_i^{(n+1)}\|^2.$$
Interpolation and numerical differentiation theories permit us to prove that
$$\left\|\left(\pd{v_i}{t}\right)^{(n+1)}-\bar\partial\tilde v_i^{(n+1)}\right\|^2\leq \left\|\left(\pd{v_i}{t}\right)^{(n+1)}-\bar\partial v_i^{(n+1)}\right\|^2+\left\|\bar\partial v_i^{(n+1)}-\bar\partial\tilde v_i^{(n+1)}\right\|^2$$
$$\leq C\delta^2\left\|\pd{{}^2v_i}{t^2}\right\|^2_{L_\infty(0,T;L_2(0,1))}+Ch^{2(k+1)}\left\|\pd{v_i}{t}\right\|^2_{L_\infty(0,T;H^{k+1}(0,1))}.$$
So,
$$\bar\partial\ \|\theta_i^{(n+1)}\|^2\leq (\frac{\gamma'_{\max}}{\gamma_0}+2)\|\theta_i^{(n+1)}\|^2 +C\delta^2\left\|\pd{{}^2v_i}{t^2}\right\|^2_{L_\infty(0,T;L_2(0,1))}$$
$$ +Ch^{2(k+1)}\left\|\pd{v_i}{t}\right\|^2_{L_\infty(0,T;H^{k+1}(0,1))} +C\left\|\pd{v_i}{y}\right\|_{L_\infty(0,T;L_2(0,1))}\frac{\gamma_1^4}{2m_a}\sum_{i=1}^{n_e}\|\theta_i^{(n+1)}\|^2 $$
$$+Ch^{2(k+1)}\left\|\pd{v_i}{y}\right\|_{L_\infty(0,T;L_2(0,1))}\frac{\gamma_1^4}{2m_a}\left\|v_i\right\|^2_{L_\infty(0,T;H^{k+1}(0,1))}$$ $$+Ch^{2(k+1)}\frac{(\alpha'_{\max}+\gamma_{\max}')^2}{2\gamma_0^2}\left\|\pd{v_i}{y}\right\|^2_{L_\infty(0,T;H^{k+1}(0,1))}.$$
Whence
$$\bar\partial\sum_{i=1}^{n_e} \|\theta_i^{(n+1)}\|^2\leq C_1\sum_{i=1}^{n_e} \|\theta_i^{(n+1)}\|^2+C_2(\delta^2+h^{2(k+1)}),$$
with $C_1=C_1(\gamma_0,\gamma_1,m_a,\gamma_{\max}',\left\|\pd{\mathbf{v}}{y}\right\|_{L_\infty(0,T;L_2(0,1))})$ and \\ $C_2=C_2(\gamma_0,\gamma_1,m_a,\gamma_{\max}',\alpha'_{\max},\left\|\pd{{}^2\mathbf{v}}{t^2}\right\|_{L_\infty(0,T;L_2(0,1))},
\left\|\pd{\mathbf{v}}{t}\right\|_{L_\infty(0,T;H^{k+1}(0,1))},$\\
$\left\|\pd{\mathbf{v}}{y}\right\|_{L_\infty(0,T;L_2(0,1))},\left\|\pd{\mathbf{v}}{y}\right\|_{L_\infty(0,T;H^{k+1}(0,1))},\left\|\mathbf{v}\right\|_{L_\infty(0,T;H^{k+1}(0,1))})$.\\
Hence
$$(1-\delta C_1)\sum_{i=1}^{n_e} \|\theta_i^{(n+1)}\|^2\leq \sum_{i=1}^{n_e} \|\theta_i^{(n)}\|^2+C_2\delta(\delta^2+h^{2(k+1)}).$$
If $\delta$ is sufficiently small, then, iterating, we obtain
$$\sum_{i=1}^{n_e} \|\theta_i^{(n+1)}\|^2\leq C_4\sum_{i=1}^{n_e} \|\theta_i^{(0)}\|^2+C_3\delta(\delta^2+h^{2(k+1)}).$$
The estimates of $\|\theta_i^{(0)}\|$ and $\|\rho_i^{(n)}\|$ complete the proof.
\end{proof}
Obtaining the solution of (\ref{eq_E_im}) implies using an iterative method
in each time step. We could apply Newton's method or some secant method, but we choose the fixed point method.
For the solution of equation (\ref{eq_E_im}), in each time step, we propose the following iterative scheme:
\begin{equation} \label{pf_E_im}
(M+\delta Aa_i(\mathbf{V}_{k}^{(n+1)})-\delta B)V_{i,k+1}^{(n+1)}=MV_i^{(n)}+\delta G_i, \quad i=1,\dots, n_e, \quad k=1,2,\dots
\end{equation}
with $\mathbf{V}_{0}^{(n+1)}=\mathbf{V}^{(n)}$ and iterating until $\|\mathbf{V}_{k+1}^{(n+1)}-\mathbf{V}_{k}^{(n+1)}\|\leq tol$.
Finally we only need to prove that this scheme converges, that is, for a prescribed $tol>0$ there exists a $K\in\mathbb{N}$ such that $\|\mathbf{V}_{k+1}^{(n+1)}-\mathbf{V}_{k}^{(n+1)}\|\leq tol$ for all $k\geq K$.
\begin{theorem}
If $\delta$ is sufficiently  small,  then the iterative scheme (\ref{pf_E_im}) converges.
\end{theorem}

\begin{proof}
The matrices $M$ and $A$ are positive definite, so, if $\delta$ is small, then system (\ref{pf_E_im}) has a unique solution for any $k=1,2,\dots$.
Subtracting the systems in two consecutive iterations, say $k$ and $k+1$, we obtain
$$(M+\delta Aa_i(\mathbf{V}_{k}^{(n+1)})-\delta B)(V_{i,k+1}^{(n+1)}-V_{i,k}^{(n+1)})$$
$$=\delta A\left(a_i(\mathbf{V}_{k}^{(n+1)})-a_i(\mathbf{V}_{k-1}^{(n+1)})\right)V_{i,k}^{(n+1)}, \quad i=1,\dots, n_e.$$
Taking  the norm on both sides of this equality, and defining $E_{i,k+1}= V_{i,k+1}^{(n+1)}-V_{i,k}^{(n+1)}$, we arrive at
$$\|M+\delta Aa_i(\mathbf{V}_{k}^{(n+1)})-\delta B\|\|E_{i,k+1}\|\leq \delta C\sum_{j=1}^{n_e}\|E_{j,k}\|.$$
For a small $\delta$, there exists a constant $C_1>0$ such that
$$\|M+\delta Aa_i(\mathbf{V}_{k}^{(n+1)})-\delta B\|\geq C_1,\quad i=1,\dots, n_e.$$
Summing up for $j=1,\dots,n_e$, the inequality becomes
$$\sum_{j=1}^{n_e}\|E_{i,k+1}\|\leq \frac{\delta C}{C_1}\sum_{j=1}^{n_e}\|E_{j,k}\|.$$
Iterating,
$$\sum_{j=1}^{n_e}\|E_{i,k+1}\|\leq \left(\frac{\delta C}{C_1}\right)^{k+1}\sum_{j=1}^{n_e}\|E_{j,0}\|.$$
If we choose the time step $\delta$ such that $\frac{\delta C}{C_1}< 1$ then, for any $tol>0$, there exists a $K$ such that for all $k>K$, $\|\mathbf{V}_{k+1}^{(n+1)}-\mathbf{V}_{k}^{(n+1)}\|\leq tol$.
\end{proof}

\subsection{Crank-Nicolson method}

The Crank-Nicolson method evaluates equation (\ref{probsd}) at the points $t_{n-1/2}=\frac{t_n+t_{n-1}}2$, $n=1,\dots, n_i$, and uses the approximations
\begin{equation*}
\mathbf{V}(y,t_{n-1/2})\approx\frac{\mathbf{V}^{(n)}(y)+\mathbf{V}^{(n-1)}(y)}{2}=\hat{\mathbf{V}}^{(n)}(y)
\end{equation*}%
and
\begin{equation*}
\frac{\partial \mathbf{V}}{\partial t}(y,t_{n-1/2})\approx \frac{\mathbf{V}^{(n)}(y)-\mathbf{V}
^{(n-1)}(y)}{\delta }=\overline{\partial }\mathbf{V}^{(n)}(y).
\end{equation*}
Then we have the problem of finding $\mathbf{V}^{(n)}\in (S_{h}^{k})^{{n_e}}$
such that it is zero on the boundary of $\Omega $, satisfies $
V_{i}^{(0)}=I_{h}(v_{i0})$, $i=1,\dots ,{n_e}$, and
\begin{equation*}
\int_{0}^{1}\overline{\partial }V_{i}^{(n)}W_{i}\ dy+a_{i}(l(\hat{V}
_{1}^{(n)}),\dots ,l(\hat{V}_{{n_e}}^{(n)}))b_{2}^{(n-1/2)}\int_{0}^{1}\frac{
\partial \hat{V}_{i}^{(n)}}{\partial y}\frac{\partial W_{i}}{\partial y}\ dy
\end{equation*}
\begin{equation}\label{eq_CN_im}
-\int_{0}^{1}b_{1}^{(n-1/2)}\frac{\partial \hat{V}_{i}^{(n)}}{\partial y}
W_{i}\ dy=\int_{0}^{1}g_{i}^{(n-1/2)}W_{i}\ dy.
\end{equation}
System (\ref{eq_CN_im}) is a non linear algebraic system due to the presence of
\linebreak $a_{i}(l(\hat{V}_{1}^{(n)}),\dots ,l(\hat{V}_{{n_e}}^{(n)}))$.

\begin{theorem}
For each $n=0,\dots, n_i-1$ and $i=1,\dots,n_e$, system (\ref{eq_CN_im}) has a solution.
\end{theorem}

\begin{proof}
The proof is similar to that of Theorem \ref{exi_E_im}.
Let $n\geq0$ be fixed. For each $h,\delta>0$, we define the continuous mapping $F:S_{h}^{k}\to S_{h}^{k}$ by
$$(F(V),W)=(V,W)-(V_0,W)+\frac{\delta}{2}b_2^{(n-\frac12)}a_i^{(n-\frac12)}(V,V_0)\left(\pd{V}{y},\pd{W}{y}\right)$$
$$ +\frac{\delta}{2}b_2^{(n-\frac12)}a_i^{(n-\frac12)}(V,V_0)\left(\pd{V_0}{y},\pd{W}{y}\right) -\frac{\delta}{2}\left(b_1\pd{V}{y},\pd{W}{y}\right) -\frac{\delta}{2}\left(b_1\pd{V_0}{y},\pd{W}{y}\right)-\delta(g,W).$$
If $W=V$ then
$$(F(V),V)=(V,V)-(V_0,V)+\frac{\delta}{2}b_2^{(n-\frac12)}a_i^{(n-\frac12)}(V,V_0)\left(\pd{V}{y},\pd{V}{y}\right) $$
$$+\frac{\delta}{2}b_2^{(n-\frac12)}a_i^{(n-\frac12)}(V,V_0)\left(\pd{V_0}{y},\pd{V}{y}\right)-\frac{\delta}{2}\left(b_1\pd{V}{y},\pd{V}{y}\right) -\frac{\delta}{2}\left(b_1\pd{V_0}{y},\pd{V}{y}\right)-\delta(g,V).$$
Thus
$$(F(V),V)\geq\|V\|^2-\|V_0\|\|V\|+\frac{\delta m_a C}{2\gamma_1^2}\|V\|^2+\frac{\delta m_a C}{2\gamma_1^2}\left(\pd{V_0}{y},\pd{V}{y}\right)+\frac{\delta\gamma'_{\max}}{4\gamma_0}\|V\|^2$$
$$-\frac{\delta}{2}\left\|b_1\pd{V_0}{y}\right\|\left\|\pd{V}{y}\right\| -\delta\|g\|\|V\|$$
$$\phantom{(F(V),V)}\geq\|V\|^2-\|V_0\|\|V\|+\frac{\delta m_a C}{2\gamma_1^2}\|V\|^2-\frac{\delta m_a C}{2\gamma_1^2}\left\|\pd{V_0}{y}\right\|\left\|\pd{V}{y}\right\|+\frac{\delta\gamma'_{\max}}{4\gamma_0}\|V\|^2$$
$$-\frac{\delta}{2}\left\|b_1\pd{V_0}{y}\right\|\left\|\pd{V}{y}\right\| -\delta\|g\|\|V\|$$
$$\phantom{(F(V),V)}\geq\|V\|^2-\|V_0\|\|V\|+\frac{\delta m_a C}{2\gamma_1^2}\|V\|^2-\frac{\delta h^{-1} m_a C}{2\gamma_1^2}\left\|\pd{V_0}{y}\right\|\left\|V\right\|+\frac{\delta\gamma'_{\max}}{4\gamma_0}\|V\|^2$$
$$-\frac{\delta h^{-1}}{2}\left\|b_1\pd{V_0}{y}\right\|\left\|V\right\| -\delta\|g\|\|V\|$$
$$\phantom{(F(V),V)}\geq\|V\|(\|V\|-\|V_0\|+\frac{\delta m_a C}{2\gamma_1^2}\|V\|-\frac{\delta h^{-1} m_a C}{2\gamma_1^2}\left\|\pd{V_0}{y}\right\|+\frac{\delta\gamma'_{\max}}{4\gamma_0}\|V\|$$
$$-\frac{\delta h^{-1}}{2}\left\|b_1\pd{V_0}{y}\right\|-\delta\|g\|).$$
$$\left(1+\delta\left(\frac{m_a C}{2\gamma_1^2}+\frac{\delta\gamma'_{\max}}{4\gamma_0}\right)\right)\|V\|\geq \|V_0\|+\delta h^{-1}\left(\frac{m_a C}{2\gamma_1^2}\left\|\pd{V_0}{y}\right\|+\frac12\left\|b_1\pd{V_0}{y}\right\|\right)$$
$$+\delta\|g\|$$
$$\Leftrightarrow\|V\|\geq \frac{4\gamma_1^2\gamma_0}{4\gamma_1^2\gamma_0+2\delta m_aC\gamma_0+\delta \gamma_1^2\gamma'_{\max}}\left(\frac{m_a C}{2\gamma_1^2}\left\|\pd{V_0}{y}\right\|+\frac12\left\|b_1\pd{V_0}{y}\right\|\right.$$
$$\left.+\frac{h}{\delta}\|V_0\|+h\|g\|\right)$$

Let us define $$\varepsilon>\frac{4\gamma_1^2\gamma_0}{4\gamma_1^2\gamma_0+2\delta m_aC\gamma_0+\delta \gamma_1^2\gamma'_{\max}}\left(\frac{m_a C}{2\gamma_1^2}\left\|\pd{V_0}{y}\right\|+\frac12\left\|b_1\pd{V_0}{y}\right\|+\frac{h}{\delta}\|V_0\|+h\|g\|\right)$$
and
$$B_{\varepsilon}=\{W\in S_{h}^{k}: \|W\|\leq \varepsilon\}.$$

Since $(F(V),V)>0$, for every $V\in\partial B_{\varepsilon}$, the corollary to Brower's Fixed Point Theorem implies the existence of a solution to Problem (\ref{eq_CN_im}).
\end{proof}
The stability is proved in the next theorem.
\begin{theorem}\label{est_CN_im}
Suppose that $\delta$ satisfies
\begin{equation}\label{cond_d_2}
\delta\leq \frac{4\gamma_0}{\gamma'_{\max}+\gamma_0}.
\end{equation}
If $\mathbf{V}^{(n)}(y)$ is the solution of (\ref{eq_CN_im}), then
$$\left\|V_i^{(n)}\right\|^2\leq C_1^{n}\left\|V_i^{(0)}\right\|^2 +\sum_{l=0}^{n}C_3^{n-l+2}\delta\left\|g_i^{(l-\frac12)}\right\|^2,$$
where $C_1$, $C_3$ could depend on  $\gamma'_{\max}$, $\gamma_0$ and $\delta$.
\end{theorem}

\begin{proof}
Putting $W_i=\hat V_i^{(n)}$ in (\ref{eq_CN_im}), we obtain
$$(\bar\partial V_i^{(n)},\hat V_i^{(n)})+a_{i}(l(\hat{V}
_{1}^{(n)}),\dots ,l(\hat{V}_{{n_e}}^{(n)}))b_{2}^{(n-1/2)}\es{\frac{
\partial \hat{V}_{i}^{(n)}}{\partial y}}{\frac{\partial \hat V_i^{(n)}}{\partial y}}$$
$$-\es{b_{1}^{(n-1/2)}\frac{\partial \hat{V}_{i}^{(n)}}{\partial y}}{
\hat V_i^{(n)}}=\es{g_{i}^{(n-1/2)}}{\hat V_i^{(n)}}.$$
Since the second term on the left-hand side is non-negative, applying Green's Theorem to the first term on the right-hand side, we obtain
$$\frac12\bar\partial \|V_i^{(n)}\|^2\leq \frac{(\gamma')^{(n-\frac12)}}{2\gamma^{(n-\frac12)}}\|\hat V_i^{(n)}\|^2+(g^{(n-\frac12)},\hat V_i^{(n)})$$
$$\leq \frac{\gamma'_{\max}}{8\gamma_0}\left(\|V_i^{(n)}\|^2+\|V_i^{(n-1)}\|^2\right)+\frac12\|g^{(n-\frac12)}\|^2+
\frac18\left(\| V_i^{(n)}\|^2+\| V_i^{(n-1)}\|^2\right).$$
So,
$$\|V_i^{(n)}\|^2\leq \|V_i^{(n-1)}\|^2+\frac{\delta\gamma'_{\max}}{4\gamma_0}\|V_i^{(n)}\|^2+
\frac{\delta\gamma'_{\max}}{4\gamma_0}\|V_i^{(n-1)}\|^2+\delta\|g^{(n-\frac12)}\|^2+
\frac{\delta}{4}\| V_i^{(n)}\|^2$$
$$+\frac{\delta}{4}\| V_i^{(n-1)}\|^2.$$
Collecting the terms, the last inequality becomes
$$\left(1-\frac{\delta\gamma'_{\max}}{4\gamma_0}-\frac{\delta}{4}\right)\|V_i^{(n)}\|^2\leq \left(1+\frac{\delta\gamma'_{\max}}{4\gamma_0}+\frac{\delta}{4}\right)\|V_i^{(n-1)}\|^2+\delta\|g^{(n-\frac12)}\|^2.$$
If delta satisfies (\ref{cond_d_2}), then
$$\|V_i^{(n)}\|^2\leq C_1\|V_i^{(n-1)}\|^2+C_2\delta\|g^{(n-\frac12)}\|^2.$$
Iterating, we obtain the desired estimate.
\end{proof}
We note that condition (\ref{cond_d_2}) permits larger step sizes in time than condition (\ref{cond_d_1}).

\begin{theorem}
If $\delta\approx h^2$ is sufficiently small, then the solution of (\ref{eq_CN_im}) is unique.
\end{theorem}

\begin{proof}
For a fixed $n$, suppose that $\mathbf{V}^{(n-1)}$ is known and that system (\ref{eq_CN_im}) has two different solutions, $\mathbf{X}$ and $\mathbf{Y}$. Subtracting both equations, we obtain
$$(X_i-Y_i,W_i)+\frac{\delta}{2}b_2^{(n-\frac12)}a_i^{(n-\frac12)}\left(\frac{\mathbf{X}+\mathbf{V}^{(n-1)}}{2}\right)
\left(\pd{X_i}{y},\pd{W_i}{y}\right)$$
$$-\frac{\delta}{2}b_2^{(n-\frac12)}a_i^{(n-\frac12)}\left(\frac{\mathbf{X}+\mathbf{V}^{(n-1)}}{2}\right)
\left(\pd{X_i}{y},\pd{W_i}{y}\right)-\frac{\delta}{2}\left(b_1^{(n-\frac12)}\pd{(X_i-Y_i)}{y},W_i\right)$$
$$+\frac{\delta}{2}b_2^{(n-\frac12)}\left(a_i^{(n-\frac12)}\left(\frac{\mathbf{X}+\mathbf{V}^{(n-1)}}{2}\right)
-a_i^{(n-\frac12)}\left(\frac{\mathbf{Y}+\mathbf{V}^{(n-1)}}{2}\right)\right)$$
$$\times\left(\pd{V_i^{(n-1)}}{y},\pd{W_i}{y}\right)=0.$$
Defining $\mathbf{E}=\mathbf{X}-\mathbf{Y}$, we can prove that
$$(E_i,W_i)+\frac{\delta}{2}b_2^{(n-\frac12)}a_i^{(n-\frac12)}\left(\frac{\mathbf{X}+\mathbf{V}^{(n-1)}}{2}\right)
\left(\pd{E_i}{y},\pd{W_i}{y}\right)-\frac{\delta}{2}\left(b_1^{(n-\frac12)}\pd{(E_i)}{y},W_i\right)$$
$$+\frac{\delta}{2} b_2^{(n-\frac12)}\left(a_i^{(n-\frac12)}\left(\frac{\mathbf{X}+\mathbf{V}^{(n-1)}}{2}\right)
-a_i^{(n-\frac12)}\left(\frac{\mathbf{Y}+\mathbf{V}^{(n-1)}}{2}\right)\right)$$
$$\times\left(\pd{(V_i^{(n-1)}+Y_i)}{y},\pd{W_i}{y}\right)=0.$$
Setting $W_i=E_i$ and applying Green's Theorem, we arrive at
$$\|E_i\|^2+\frac{\delta}{2}b_2^{(n-\frac12)}a_i^{(n-\frac12)}\left(\frac{\mathbf{X}+\mathbf{V}^{(n-1)}}{2}\right)
\left\|\pd{E_i}{y}\right\|^2$$
$$=\frac{\delta(\gamma')^{(n-\frac12)}}{4\gamma^{(n-\frac12)}}\|E_i\|^2+
\frac{\delta}{2} b_2^{(n-\frac12)}\left(a_i^{(n-\frac12)}\left(\frac{\mathbf{X}+\mathbf{V}^{(n-1)}}{2}\right)
\right.$$
$$\left.-a_i^{(n-\frac12)}\left(\frac{\mathbf{Y}+\mathbf{V}^{(n-1)}}{2}\right)\right)\left(\pd{(V_i^{(n-1)}+Y_i)}{y},\pd{E_i}{y}\right).$$
Then
$$\|E_i\|^2+\frac{\delta m_a}{2\gamma_1^2}\left\|\pd{E_i}{y}\right\|^2
=\frac{\delta \gamma'_{\max}}{4\gamma_0}\|E_i\|^2+
\frac{\delta\gamma_1^2C}{8m_a} \left\|\pd{(V_i^{(n-1)}+Y_i)}{y}\right\|_\infty\sum_{j=1}^{n_e}\|E_j\|^2$$
$$+\frac{\delta m_a}{2\gamma_1^2}\left\|\pd{E_i}{y}\right\|^2,$$
and so
$$\left(1-\frac{\delta \gamma'_{\max}}{4\gamma_0}-\frac{\delta\gamma_1^2C}{8m_a} \left\|\pd{(V_i^{(n-1)}+Y_i)}{y}\right\|_\infty\right)\sum_{j=1}^{n_e}\|E_j\|^2\leq 0.$$
As before, we have
$$\left\|\pd{(V_i^{(n-1)}+Y_i)}{y}\right\|_\infty\leq Ch^{-2}(\|V_i^{(n-1)}\|+\|Y_i\|).$$
By Theorem \ref{est_CN_im}, the result is proved, provided that  $\delta\approx h^2$ is sufficiently small.
\end{proof}

\begin{theorem}\label{conv_CN_im}
If $v$ is a solution of equation (\ref{prob1}) and $V_n$ is a solution of (\ref{eq_CN_im}), then
\begin{equation*}
\|V_i^{(n)}(y)-v_i(y,t_n)\| \leq C(h^{k+1}+\delta^2),\quad n=1,\dots,n_t, \quad i=1,\dots,n_e,
\end{equation*}
for a certain $\delta$ and $C=C\left(M_a,m_a,\gamma_0,\gamma'_{\max},\alpha'_{\max},\left\|\pd vy\right\|_{L_{\infty}(0,T,L_2(\Omega))},\right.$\\
$\left.\|v\|_{L_{\infty}(0,T,H^{k+1}(\Omega))},\left\|\pd vt\right\|_{L_{\infty}(0,T,L_2(\Omega))},\left\|\pd{^2 v}{t^2}\right\|_{L_{\infty}(0,T,L_2(\Omega))},\left\|\pd{^3 v}{t^3}\right\|_{L_{\infty}(0,T,L_2(\Omega))},\right.$\\
$\left.\left\|\pd{^3 v}{y\partial t^2}\right\|_{L_{\infty}(0,T,L_2(\Omega))}\right)$ which doesn't depend on $h$, $k$ and $\delta$.
\end{theorem}

\begin{proof}
We have
$$(\bar\partial\theta_i^{(n)},W_i)+b_2^{(n-\frac12)}a_i^{(n-\frac12)}(\hat{\mathbf{V}}^{(n)})\left(\pd{\hat \theta_i^{(n)}}{y},\pd{W_i}{y}\right)-\left(b_1^{(n-\frac12)}\pd{\hat \theta_i^{(n)}}{y},W_i\right)$$
$$=(\bar\partial V_i^{(n)},W_i)+b_2^{(n-\frac12)}a_i^{(n-\frac12)}(\hat{\mathbf{V}}^{(n)})\left(\pd{\hat V_i^{(n)}}{y},\pd{W_i}{y}\right)-\left(b_1^{(n-\frac12)}\pd{\hat V_i^{(n)}}{y},W_i\right)$$
$$-(\bar\partial \tilde v_i^{(n)},W_i)-b_2^{(n-\frac12)}a_i^{(n-\frac12)}(\hat{\mathbf{V}}^{(n)})\left(\pd{\hat{ \tilde{ v}}_i^{(n)}}{y},\pd{W_i}{y}\right)-\left(b_1^{(n-\frac12)}\pd{\hat{ \tilde{ v}}_i^{(n)}}{y},W_i\right)$$
$$=(g_i^{(n)},W_i)-(\bar\partial \tilde v_i^{(n)},W_i)- b_2^{(n-\frac12)}a_i^{(n-\frac12)}(\hat{\mathbf{V}}^{(n)})\left(\pd{\hat{ \tilde{ v}}_i^{(n)}}{y},\pd{W_i}{y}\right)$$
$$+\left(b_1^{(n-\frac12)}\pd{\hat{ \tilde{ v}}_i^{(n)}}{y},W_i\right)$$
$$=\left(\left(\pd{v_i}{t}\right)^{(n-\frac12)},W_i\right)+b_2^{(n-\frac12)}a_i^{(n-\frac12)}(\mathbf{v}^{(n-\frac12)})\left(\pd{ v_i^{(n-\frac12)}}{y},\pd{W_i}{y}\right)$$
$$-\left(b_1^{(n-\frac12)}\pd{v_i^{(n-\frac12)}}{y},W_i\right)-(\bar\partial \tilde v_i^{(n)},W_i)- b_2^{(n-\frac12)}a_i^{(n-\frac12)}(\hat{\mathbf{V}}^{(n)})\left(\pd{\hat{ v}_i^{(n)}}{y},\pd{W_i}{y}\right)$$
$$+\left(b_1^{(n-\frac12)}\pd{\hat{ \tilde{ v}}_i^{(n)}}{y},W_i\right)$$
$$=\left(\left(\pd{v_i}{t}\right)^{(n-\frac12)}-\bar\partial \tilde v_i^{(n)},W_i\right)+ b_2^{(n-\frac12)}a_i^{(n-\frac12)}(\mathbf{v}^{(n-\frac12)})\left(\pd{v_i^{(n-\frac12)}}{y}-\pd{\hat{v}_i^{(n)}}{y},\pd{W_i}{y}\right)$$
$$+b_2^{(n-\frac12)}(a_i^{(n-\frac12)}(\mathbf{v}^{(n-\frac12)})-a_i^{(n-\frac12)}(\hat{\mathbf{V}}^{(n)}))\left(\pd{\hat{ v}_i^{(n)}}{y},\pd{W_i}{y}\right).$$
$$-\left(b_1^{(n-\frac12)}\left(\pd{v_i^{(n-\frac12)}}{y}-\pd{\hat{ \tilde{ v}}_i^{(n)}}{y}\right),W_i\right)$$
Choosing $W_i=\hat\theta_i^{(n)}$, we arrive at
$$\frac12\bar\partial\|\theta_i^{(n)}\|^2+ b_2^{(n-\frac12)}a_i^{(n-\frac12)}(\hat{\mathbf{V}}^{(n)})\left\|\pd{\hat \theta_i^{(n)}}{y}\right\|^2=\left(b_1^{(n-\frac12)}\pd{\hat \theta_i^{(n)}}{y},\hat \theta_i^{(n)}\right)$$
$$ +\left(\left(\pd{v_i}{t}\right)^{(n-\frac12)}-\bar\partial \tilde v_i^{(n)},\hat\theta_i^{(n)}\right) +b_2^{(n-\frac12)}a_i^{(n-\frac12)}(\mathbf{v}^{(n-\frac12)})\left(\pd{v_i^{(n-\frac12)}}{y}-\pd{\hat{v}_i^{(n)}}{y},\pd{\hat\theta_i^{(n)}}{y}\right)$$
$$+b_2^{(n-\frac12)}(a_i^{(n-\frac12)}(\mathbf{v}^{(n-\frac12)})-a_i^{(n-\frac12)}(\hat{\mathbf{V}}^{(n)}))\left(\pd{\hat{ v}_i^{(n)}}{y},\pd{\hat\theta_i^{(n)}}{y}\right).$$
$$ -\left(b_1^{(n-\frac12)}\left(\pd{v_i^{(n-\frac12)}}{y}-\pd{\hat{ \tilde{ v}}_i^{(n)}}{y}\right),\hat\theta_i^{(n)}\right)$$
Integrating by parts, we obtain
$$\left(b_1^{(n-\frac12)}\pd{\hat \theta_i^{(n)}}{y},\hat\theta_i^{(n)}\right)= -\frac{\left(\gamma'\right)^{(n-\frac12)}}{2\gamma^{(n-\frac12)}}\|\hat \theta_i^{(n)}\|^2,$$
$$\left(b_1^{(n-\frac12)}\left(\pd{v_i^{(n-\frac12)}}{y}-\pd{\hat{\tilde{v}}_i^{(n)}}{y}\right),\hat\theta_i^{(n)}\right)= -\frac{\left(\gamma'\right)^{(n-\frac12)}}{2\gamma^{(n-\frac12)}}( v_i^{(n-\frac12)}-\hat{\tilde{v}}_i^{(n)},\hat\theta_i^{(n)})$$
$$-\left(b_1^{(n-\frac12)}(v_i^{(n-\frac12)}-\hat{\tilde{v}}_i^{(n)}),\pd{\hat\theta_i^{(n)}}{y}\right).$$
Applying the H\"older and Cauchy inequalities, we obtain the inequality
$$\frac12\bar\partial\|\theta_i^{(n)}\|^2+ \frac{m_a}{\gamma_0^2}\left\|\pd{\hat \theta_i^{(n)}}{y}\right\|^2\leq \frac{\gamma'_{\max}}{2\gamma_0}\|\hat \theta_i^{(n)}\|^2+ C\left(\left\|\left(\pd{v_i}{t}\right)^{(n-\frac12)}-\bar\partial \tilde v_i^{(n)}\right\|^2\right.$$
$$\left. +\left\|\pd{v_i^{(n-\frac12)}}{y}-\pd{\hat{v}_i^{(n)}}{y}\right\|^2 +\|v_i^{(n-\frac12)}-\hat{\tilde{v}}_i^{(n)}\|^2+|a_i^{(n-\frac12)}(\mathbf{v}^{(n-\frac12)}) -a_i^{(n-\frac12)}(\hat{\mathbf{V}}^{(n)})|^2\right)$$
$$ +\frac{m_a}{\gamma_0^2}\left\|\pd{\hat \theta_i^{(n)}}{y}\right\|^2,$$
where $C=C(M_a,m_a,\gamma_0,\gamma'_{\max},\alpha'_{\max},\|\pd{v_i}{y}\|_{L_{\infty}(0,T,L_2(0,1))})$.
Using interpolation and differentiation theory we can establish the following estimates:
$$\left\|\left(\pd{v_i}{t}\right)^{(n-\frac12)}-\bar\partial \tilde v_i^{(n)}\right\|\leq \left\|\left(\pd{v_i}{t}\right)^{(n-\frac12)}-\bar\partial v_i^{(n)}\right\| +\left\|\bar\partial v_i^{(n)}-\bar\partial \tilde v_i^{(n)}\right\|$$
$$\leq C\delta\left\|\pd{{}^3v_i}{t^3}\right\|_{L_{\infty}(0,T,L_2(0,1))} +Ch^{k+1}\|v_i\|_{L_{\infty}(0,T,H^{k+1}(0,1))};$$
$$\left\|\pd{v_i^{(n-\frac12)}}{y}-\pd{\hat{v}_i^{(n)}}{y}\right\|\leq C\delta^2\left\|\pd{{}^3v_i}{y\partial t^2}\right\|_{L_{\infty}(0,T,L_2(0,1))};$$
$$\|v_i^{(n-\frac12)}-\hat{\tilde{v}}_i^{(n)}\|\leq\|v_i^{(n-\frac12)}-\tilde{v}_i^{(n-\frac12)}\| +\|\tilde{v}_i^{(n-\frac12)}-\hat{\tilde{v}}_i^{(n)}\|$$
$$\leq Ch^{k+1}\|v_i\|_{L_{\infty}(0,T,H^{k+1}(0,1))}+C\delta^2\left\|\pd{{}^2v_i}{ t^2}\right\|_{L_{\infty}(0,T,L_2(0,1))};$$
$$|a_i^{(n-\frac12)}(\mathbf{v}^{(n-\frac12)}) -a_i^{(n-\frac12)}(\hat{\mathbf{V}}^{(n)})|$$
$$\leq \sum_{j=1}^{n_e}C_i\|v_j^{(n-\frac12)}-\hat{V}_j^{(n)}\|\leq C\sum_{j=1}^{n_e}\|v_j^{(n-\frac12)}-\hat{v}_j^{(n)}\| +\|\hat{v}_j^{(n)}-\hat{V}_j^{(n)}\| $$
$$\leq C\sum_{j=1}^{n_e}\delta^2\left\|\pd{{}^2v_j}{ t^2}\right\|_{L_{\infty}(0,T,L_2(0,1))} +\|\hat{\theta}_j^{(n)}\| +\|\hat{\rho}_j^{(n)}\|.$$
So,
$$ \frac12\bar\partial\|\theta_i^{(n)}\|^2\leq \frac{\gamma'_{\max}}{2\gamma_0}\|\hat \theta_i^{(n)}\|^2 +C(\delta^2+h^{k+1})^2 +C_1\sum_{j=1}^{n_e}\|\hat{\theta}_j^{(n)}\| +\|\hat{\rho}_j^{(n)}\|$$
but, in this inequality, $C=C(T,M_a,m_a,\gamma_0,\gamma'_{\max},\alpha'_{\max},\|v_i\|_{L_{\infty}(0,T,H^{k+1}(0,1))},$\\
$ \|\pd{v_i}{y}\|_{L_{\infty}(0,T,L_2(0,1))}, \left\|\pd{{}^2v_i}{ t^2}\right\|_{L_{\infty}(0,T,L_2(0,1))}, \left\|\pd{{}^3v_i}{y\partial t^2}\right\|_{L_{\infty}(0,T,L_2(0,1))},$\\ $\left\|\pd{{}^3v_i}{t^3}\right\|_{L_{\infty}(0,T,L_2(0,1))}).$
Thus
$$\|\theta_i^{(n)}\|^2\leq C\frac{\delta\gamma'_{\max}}{2\gamma_0}\|\theta_i^{(n)}\|^2 +\left(1+\frac{\delta\gamma'_{\max}}{2\gamma_0}\right)\|\theta_i^{(n-1)}\|^2 +C\delta(\delta^2+h^{k+1})^2 $$
$$+C\delta\sum_{j=1}^{n_e}\|\theta_j^{(n)}\|^2 +C\delta\sum_{j=1}^{n_e}\|\rho_j^{(n)}\|^2 +C\delta\sum_{j=1}^{n_e}\|\theta_j^{(n-1)}\|^2 +C\delta\sum_{j=1}^{n_e}\|\rho_j^{(n-1)}\|^2. $$
Summing for $i=1,\dots,n_e$ and recalling the estimate for $\rho$, we obtain
$$\left(1-C\delta-\frac{\delta\gamma'_{\max}}{2\gamma_0}\right)\sum_{j=1}^{n_e}\|\theta_j^{(n)}\|^2$$
$$\leq \left(1+C\delta+\frac{\delta\gamma'_{\max}}{2\gamma_0}\right)\sum_{j=1}^{n_e}\|\theta_j^{(n-1)}\|^2 +C\delta(\delta^2+h^{k+1})^2.$$
If $\delta$ satisfies
\begin{equation}\label{condelta0}
\delta\leq\frac{2\gamma_0}{2C\gamma_0+\gamma'_{\max}}
\end{equation}
then
$$\sum_{j=1}^{n_e}\|\theta_j^{(n)}\|^2\leq C\sum_{j=1}^{n_e}\|\theta_j^{(n-1)}\|^2 +C\delta(\delta^2+h^{k+1})^2.$$
Iterating, we arrive at
$$\sum_{j=1}^{n_e}\|\theta_j^{(n)}\|^2\leq C^n\sum_{j=1}^{n_e}\|\theta_j^{(0)}\|^2 +C(\delta^2+h^{k+1})^2.$$
Since $\|\theta_j^{(0)}\|\leq Ch^{k+1}\|v_{j0}\|_{H^{k+1}(0,1)}$, adding the estimates of $\rho_j$, the result follows.
\end{proof}

For the solution of equation (\ref{eq_CN_im}), in each time step, we propose the following iterative scheme:
$$(M+\delta Aa_i(\frac{\mathbf{V}_{k}^{(n+1)}+\mathbf{V}^{(n)}}{2})-\delta B)V_{i,k+1}^{(n+1)}$$
\begin{equation} \label{pf_CN_im}
=(M-\delta Aa_i(\frac{\mathbf{V}_{k}^{(n+1)}+\mathbf{V}^{(n)}}{2})+\delta B)V_i^{(n)}+\delta G_i,
\end{equation}
$i=1,\dots, n_e,$ $k=1,2,\dots,$ with $\mathbf{V}_{0}^{(n+1)}=\mathbf{V}^{(n)}$ and iterating until $\|\mathbf{V}_{k+1}^{(n+1)}-\mathbf{V}_{k}^{(n+1)}\|\leq tol$.

\begin{theorem}
If $\delta$ is sufficiently small then the iterative scheme (\ref{pf_CN_im}) converges.
\end{theorem}

\begin{proof}
The matrices $M$ and $A$ are positive definite, so if $\delta$ is small then system (\ref{pf_CN_im}) has a unique solution for any $k=1,2,\dots$.
Subtracting the systems in two consecutive iterations, say $k$ and $k+1$, taking the norm on both sides and defining $E_{i,k+1}= V_{i,k+1}^{(n+1)}-V_{i,k}^{(n+1)}$, we obtain
$$\|M+\delta Aa_i(\frac{\mathbf{V}_{k}^{(n+1)}+\mathbf{V}^{(n)}}{2})-\delta B\|\|E_{i,k+1}\|\leq \delta C\sum_{j=1}^{n_e}\|E_{j,k}\|.$$
For a small $\delta$, there exists a constant $C_1>0$ such that
$$\|M+\delta Aa_i(\frac{\mathbf{V}_{k}^{(n+1)}+\mathbf{V}^{(n)}}{2})-\delta B\|\geq C_1,\quad i=1,\dots, n_e,$$
and summing up for $j=1,\dots,n_e$, the inequality becomes
$$\sum_{j=1}^{n_e}\|E_{i,k+1}\|\leq \frac{\delta C}{C_1}\sum_{j=1}^{n_e}\|E_{j,k}\|.$$
Iterating, we obtain
$$\sum_{j=1}^{n_e}\|E_{i,k+1}\|\leq \left(\frac{\delta C}{C_1}\right)^{k+1}\sum_{j=1}^{n_e}\|E_{j,0}\|.$$
If we choose the time step $\delta$ such that $\frac{\delta C}{C_1}< 1$ then, for any $tol>0$, there exists a $K$ such that for all $k>K$, $\|\mathbf{V}_{k+1}^{(n+1)}-\mathbf{V}_{k}^{(n+1)}\|\leq tol$.
\end{proof}

\subsection{Linearised Crank-Nicolson method}

In order to avoid the application of an iterative method in each time step, we implement the linearised method suggested in \cite{Tho06}, substituting $\hat{V}_{i}^{(n)}$ with $\overline{V}_{i}^{(n)}=\frac{3}{2}
V_{i}^{(n-1)}-\frac{1}{2}V_{i}^{(n-2)}$ in the diffusion coefficient. So, the
totally discrete problem, in this case, will be to calculate the functions $\mathbf{V}^{(n)}$, $n\geq 2$, belonging to $(S_{h}^{k})^{{n_e}}$, which are
zero on the boundary of $\Omega $ and satisfy
\begin{equation*}
\es{\overline{\partial }V_{i}^{(n)}}{W_{i}}+a_{i}(l(\overline{V}
_{1}^{(n)}),\dots ,l(\overline{V}_{{n_e}}^{(n)}))b_{2}^{(n-1/2)}\es{\frac{
\partial \hat{V}_{i}^{(n)}}{\partial y}}{\frac{\partial W_{i}}{\partial y}}
\end{equation*}
\begin{equation}\label{descn}
-\es{b_{1}^{(n-1/2)}\frac{\partial \hat{V}_{i}^{(n)}}{\partial y}}{
W_{i}}=\es{g_{i}^{(n-1/2)}}{W_{i}},\quad n\geq 2,\quad i=1,\dots
,{n_e}.
\end{equation}
In this way, we have a linear multistep method which requires two initial
estimates $\mathbf{V}^{(0)}$ and $\mathbf{V}^{(1)}$. The estimate $\mathbf{V}
^{(0)}$ is obtained by the initial condition as $V_{i}^{(0)}=I_{h}(v_{i0})$.
In order to calculate $\mathbf{V}^{(1)}$ with the same accuracy, we follow
\cite{Tho06} and use the following predictor-corrector scheme:
\begin{equation*}
\es{\frac{V_{i}^{(1,0)}-V_{i}^{(0)}}{\delta }}{W_{i}} +a_{i}(l(V_{1}^{(0)}),\dots ,l(V_{{n_e}}^{(0)}))b_{2}^{(1/2)}\es{\frac{\partial }{\partial y}\left( \frac{V_{i}^{(1,0)}+V_{i}^{(0)}}{2}
\right)}{ \frac{\partial W_{i}}{\partial y}}
\end{equation*}
\begin{equation}\label{desc10}
-\es{b_{1}^{(1/2)}\frac{\partial }{\partial y}\left( \frac{
V_{i}^{(1,0)}+V_{i}^{(0)}}{2}\right)}{ W_{i}}=\es{g_{i}^{(1/2)}}{W_{i}},\quad i=1,\dots ,{n_e},
\end{equation}
\begin{equation*}
\es{\overline{\partial }V_{i}^{(1)}}{W_{i}}+a_{i}\left( l\left(
\frac{V_{1}^{(1,0)}+V_{1}^{(0)}}{2}\right) ,\dots ,l\left( \frac{
V_{{n_e}}^{(1,0)}+V_{{n_e}}^{(0)}}{2}\right) \right) b_{2}^{(1/2)}\es{\frac{\partial \hat{V}_{i}^{(1)}}{\partial y}}{\frac{\partial W_{i}}{\partial y}}
\end{equation*}
\begin{equation}\label{desc1}
-\es{b_{1}^{(1/2)}\frac{\partial \hat{V}_{i}^{(1)}}{\partial y}}{
W_{i}}=\es{g_{i}^{(1/2)}}{W_{i}},\quad i=1,\dots ,{n_e}.
\end{equation}
Systems (\ref{descn})-(\ref{desc1}) are all linear and for small values of $\delta$ they always have a unique solution. The proof of the stability of the solutions is similar to that of Theorem \ref{est_CN_im}.

\begin{theorem}  \label{conv_CN_ex}
If $\mathbf{v}$ is the solution of equation (\ref{prob1}) and
$\mathbf{V}^{(n)}$ is the solution of (\ref{descn})-(\ref{desc1}), then
\begin{equation*}
\Vert V_{i}^{(n)}(y)-v_{i}(y,t_{n})\Vert \leq C(h^{k+1}+\delta ^{2}),\quad
n=1,\dots ,n_i,\quad i=1,\dots ,{n_e},
\end{equation*}%
where $C$ does not depend on $h$, $k$ or $\delta $, but could depend on $M_a$, $m_a$, $\gamma_0$, $\gamma'_{\max}$, $\alpha'_{\max}$, $\left\|\pd vy\right\|_{L_{\infty}(0,T,L_2(\Omega))}$, $\|v\|_{L_{\infty}(0,T,H^{k+1}(\Omega))}$, $\left\|\pd vt\right\|_{L_{\infty}(0,T,L_2(\Omega))}$, $\left\|\pd{^2 v}{t^2}\right\|_{L_{\infty}(0,T,L_2(\Omega))}$, $\left\|\pd{^3 v}{t^3}\right\|_{L_{\infty}(0,T,L_2(\Omega))}$ and $\left\|\pd{^3 v}{y\partial t^2}\right\|_{L_{\infty}(0,T,L_2(\Omega))}$.
\end{theorem}

\begin{proof}
First, we will determine the estimate for $n=1$.
Let $\theta_i^{(1,0)}=V_i^{(1,0)}-\tilde v_i^{(1)}$, $\hat
\theta_i^{(1,0)}=\frac{\theta_i^{(1,0)}+\theta_i^{(0)}}{2}$ and
$\overline{\partial}\theta_i^{(1,0)}=\frac{\theta_i^{(1,0)}-\theta_i^{(0)}}{\delta}$. Arguing in the same way as in Theorem \ref{conv_CN_im} and setting  $W_i=\hat\theta_i^{(1,0)}$ in (\ref{desc10}), we have\\

$\displaystyle
\frac12\overline{\partial}\|\theta_i^{(1,0)}\|^2+\frac{m_a}{\gamma_0^2}\left\|\frac{\partial\hat\theta_i^{(1,0)}}{\partial y}\right\|^2\leq$
$$\leq C \left( \left\|\left(\pd{v_i}{t}\right)^{(1/2)}-\overline\partial\tilde v_i^{(1)}\right\|+\left\|\pd{v_i^{(1/2)}}{y}-\pd{\hat{v}_i^{(1)}}{y}\right\|+\sum_{j=1}^{{n_e}}\|v_{j}^{(1/2)}-V_{j}^{(0)}\|\right.$$
$$+\|\hat{\tilde v}_i^{(1)}-v_i^{(1/2)}\|\Bigg)\left\|\frac{\partial\hat\theta_i^{(1,0)}}{\partial y}\right\|.$$
Using Cauchy's inequality, it follows that
\begin{eqnarray*}
\overline{\partial}\|\theta_i^{(1,0)}\|^2&\leq& C \left( \left\|\left(\pd{v_i}{t}\right)^{(1/2)}-\overline\partial\tilde v_i^{(1)}\right\|+\left\|\pd{v_i^{(1/2)}}{y}-\pd{\hat{v}_i^{(1)}}{y}\right\|\right.\\ &&+\sum_{j=1}^{{n_e}}\|v_{j}^{(1/2)}-V_{j}^{(0)}\|+\|\hat{\tilde v}_i^{(1)}-v_i^{(1/2)}\|\Bigg),
\end{eqnarray*}
with $C=C(M_a,m_a,\gamma_0,\gamma'_{\max},\alpha'_{\max},\|\pd{v_i}{y}\|_{L_{\infty}(0,T,L_2(0,1))})$.
The following estimates are true for every $i\in\{1,\dots,{n_e}\}$,
\begin{eqnarray*}
\left\|\left(\pd{v_i}{t}\right)^{(1/2)}-\overline\partial\tilde v_i^{(1)}\right\|&\leq&\left\|\left(\pd{v_i}{t}\right)^{(1/2)}-\overline\partial v_i^{(1)}\right\|+\|\overline\partial v_i^{(1)}-\overline\partial\tilde v_i^{(1)}\|\\
&\leq &C\delta^2+C h^{k+1},\\
\left\|\pd{v_i^{(1/2)}}{y}-\pd{\hat{v}_i^{(1)}}{y}\right\|&\leq& C\delta\int_{t_0}^{t_1}\left\|\frac{\partial^3 v_{i}}{\partial y\partial t^2}\right\|\ dt\leq C\delta^2,\\
\|v_{i}^{(1/2)}-V_{i}^{(0)}\|&\leq&\|v_{i}^{(1/2)}-v_{i}^{(0)}\|+\|v_{i}^{(0)}-V_{i}^{(0)}\|\leq C\delta+C h^{k+1},\\
\end{eqnarray*}
and
$$\|\hat{\tilde v}_i^{(1)}-v_i^{(1/2)}\|\leq  \|\hat{\tilde v}_i^{(1)}-\hat{\tilde v}_i^{(1/2)}\|+\|\hat{\tilde v}_i^{(1/2)}-v_i^{(1/2)}\|\leq C\delta^2+C h^{k+1}.$$
Hence
$$\overline{\partial}\|\theta_i^{(1,0)}\|^2\leq C (h^{k+1}+\delta)^2,$$
and we have the estimate
$$\|\theta_i^{(1,0)}\|^2\leq\|\theta_i^{(0)}\|^2+ C\delta (h^{k+1}+\delta)^2\leq C(h^{2(k+1)}+\delta^3),\quad i=1,\dots,{n_e},$$
where $C=C(T,M_a,m_a,\gamma_0,\gamma'_{\max},\alpha'_{\max},\|v_i\|_{L_{\infty}(0,T,H^{k+1}(0,1))}, \|\pd{v_i}{y}\|_{L_{\infty}(0,T,L_2(0,1))},$\\  $\left\|\pd{{}^2v_i}{ t^2}\right\|_{L_{\infty}(0,T,L_2(0,1))}, \left\|\pd{{}^3v_i}{y\partial t^2}\right\|_{L_{\infty}(0,T,L_2(0,1))}, \left\|\pd{{}^3v_i}{t^3}\right\|_{L_{\infty}(0,T,L_2(0,1))}).$\\

\noindent Repeating this process for equation (\ref{desc1}), we arrive at\\

$\displaystyle
\frac12\overline{\partial}\|\theta_i^{(1)}\|^2+\frac{m_a}{\gamma_0^2}\left\|\frac{\partial\hat\theta_i^{(1)}}{\partial y}\right\|^2$
\begin{eqnarray*}
&\leq& C \left( \left\|\left(\pd{v_i}{t}\right)^{(1/2)}-\overline\partial\tilde v_i^{(1)}\right\|+\left\|\pd{v_i^{(1/2)}}{y}-\pd{\hat{v}_i^{(1)}}{y}\right\|\right.\\
&&\left.+\sum_{j=1}^{{n_e}}\left\|v_{j}^{(1/2)}-\frac{V_{j}^{(1,0)}-V_{j}^{(0)}}{2}\right\|+\|\hat{\tilde v}_i^{(1)}-v_i^{(1/2)}\|\right)\left\|\frac{\partial\hat\theta_i^{(1)}}{\partial y}\right\|.\\
\end{eqnarray*}
In this case, we use the estimate
\begin{eqnarray*}
\left\|v_{i}^{(1/2)}-\frac{V_{i}^{(1,0)}-V_{i}^{(0)}}{2}\right\|&\leq& \|v_{i}^{(1/2)}-\hat{\tilde v}_i^{(1)}\|+\|\hat{\tilde v}_i^{(1)}-\frac{V_{i}^{(1,0)}-V_{i}^{(0)}}{2}\|\\
&\leq& \|v_{i}^{(1/2)}-\hat{\tilde v}_i^{(1)}\|+\frac12\|\theta_i^{(1,0)}\|+\frac12\|\theta_i^{(0)}\|\\
&\leq& C(h^{k+1}+\delta^{2})+C h^{k+1}+C(h^{k+1}+\delta^{\frac32})\\
&\leq& C(h^{k+1}+\delta^{\frac32}),
\end{eqnarray*}
and then, by Cauchy's inequality, we conclude that
$$\overline{\partial}\|\theta_i^{(1)}\|^2\leq C (h^{2(k+1)}+\delta^3),$$
whence
$$\|\theta_i^{(1)}\|^2\leq\|\theta_i^{(0)}\|^2+ C\delta (h^{2(k+1)}+\delta^3)\leq C(h^{2(k+1)}+\delta^4).$$
To conclude the proof, we obtain the result for $n\geq2$, applying the same process to equation (\ref{descn}). In this way, we obtain\\

$\displaystyle
\frac12\overline{\partial}\|\theta_i^{(n)}\|^2+\frac{m_a}{\gamma_0^2}\left\|\frac{\partial\hat\theta_i^{(n)}}{\partial y}\right\|^2$
\begin{eqnarray*}
&\leq& C \left( \left\|\left(\pd{v_i}{t}\right)^{(n-1/2)}-\overline\partial\tilde v_i^{(n)}\right\|+\left\|\pd{v_i^{(n-1/2)}}{y}-\pd{\hat{v}_i^{(n)}}{y}\right\|\right.\\
&&\left.+\sum_{j=1}^{{n_e}}\left\|v_{j}^{(n-1/2)}-\bar{V}_j^{(n)}\right\|+\|\hat{\tilde v}_i^{(n)}-v_i^{(n-1/2)}\|\right)\left\|\frac{\partial\hat\theta_i^{(n)}}{\partial y}\right\|.
\end{eqnarray*}
Now, we need the estimate
\begin{eqnarray*}
\left\|v_{i}^{(n-1/2)}-\bar{V}_i^{(n)}\right\|&\leq& \|v_{i}^{(n-1/2)}-\bar{v}_i^{(n)}\|+\|\bar{v}_i^{(n)}-\bar{V}_i^{(n)}\|\\
&\leq& \|v_{i}^{(n-1/2)}-\bar{v}_i^{(n)}\|+\|\overline\rho_i^{(n)}\|+\|\overline\theta_i^{(n)}\|\\
&\leq& C\delta^2+C h^{k+1}+C(\|\theta_{n-1}\|+\|\theta_{n-2}\|)
\end{eqnarray*}
to prove that
$$\overline{\partial}\|\theta_i^{(n)}\|^2\leq C\sum_{j=1}^{{n_e}}\|\theta_j^{(n-1)}\|^2+C\sum_{j=1}^{{n_e}}\|\theta_j^{(n-2)}\|^2+C(h^{(k+1)}+\delta^2)^2,\quad i=1,\dots,{n_e}.$$
Summing up for all $i$, it follows that
$$\overline{\partial}\sum_{i=1}^{{n_e}}\|\theta_i^{(n)}\|^2\leq C\sum_{j=1}^{{n_e}}\|\theta_j^{(n-1)}\|^2+C\sum_{j=1}^{{n_e}}\|\theta_j^{(n-2)}\|^2+C(h^{(k+1)}+\delta^2)^2.$$
Iterating, we obtain
$$\sum_{i=1}^{{n_e}}\|\theta_i^{(n)}\|^2\leq(1+C\delta)\sum_{i=1}^{{n_e}}\|\theta_i^{(n-1)}\|^2+ C\delta\sum_{i=1}^{{n_e}}\|\theta_i^{(n-2)}\|^2 +C\delta (h^{k+1}+\delta^2)^2$$
$$\leq C\sum_{i=1}^{{n_e}}\|\theta_i^{(1)}\|^2+ C\sum_{i=1}^{{n_e}}\delta\|\theta_i^{(0)}\|^2+ C\delta (h^{k+1}+\delta^2)^2$$
and, recalling the estimates for $\|\theta_i^{(0)}\|$, $\|\theta_i^{(1)}\|$ and
$\|\rho_i^{(n)}\|$, the proof is complete.
\end{proof}

The conditions on $h$, $\delta$, $\gamma'_{\max}$ and $\gamma_0$ are the same as those in Theorem \ref{conv_CN_im}.


\section{Example}

The final step is to implement this method using a programming language. To
perform this task, we choose the Matlab environment. In this section, we present one example to illustrate the applicability and robustness of the methods, comparing the results with the theoretical
results proved and with the results obtained with the method presented in \cite{RACF14}.
We simulate a problem with a known exact solution, which
will permit us to calculate the error and confirm numerically the
theoretical convergence rates. Let us consider Problem (\ref{prob0}) with
two equations in $Q_{t}$ and $T=1$. The diffusion
coefficients are
\begin{equation*}
a_1(r,s)=2-\frac{1}{1+r^2}+\frac{1}{1+s^2},\quad a_2(r,s)=3+\frac{2}{1+r^2}-%
\frac{1}{1+s^2},
\end{equation*}
the movement of the boundaries is given by the functions
\begin{equation*}
\alpha(t)=-\frac{t}{1+t},\quad \beta(t)=1+\frac{2t}{1+t},
\end{equation*}
the functions $f_1(x,t)$, $f_2(x,t)$, $u_{10}(x,t)$ and $u_{20}(x,t)$ are
chosen such that
\begin{equation*}
u_1(x,t)=\frac{1}{t+1}\left(\frac{611}{70}z-\frac{10513}{210}z^2+\frac{646}{7%
}z^3-\frac{1070}{21}z^4\right)
\end{equation*}
and
\begin{equation*}
u_2(x,t)=e^{-t}\left(\frac{2047}{140}z-\frac{27701}{420}z^2+\frac{691}{7}z^3-%
\frac{995}{21}z^4\right)
\end{equation*}
with exact solutions
\begin{equation*}
z=\frac{(2t+1)(x+tx+t)}{5t^2+5t+1}.
\end{equation*}

\begin{figure}[!htb]
\center
\includegraphics[width=0.9\textwidth]{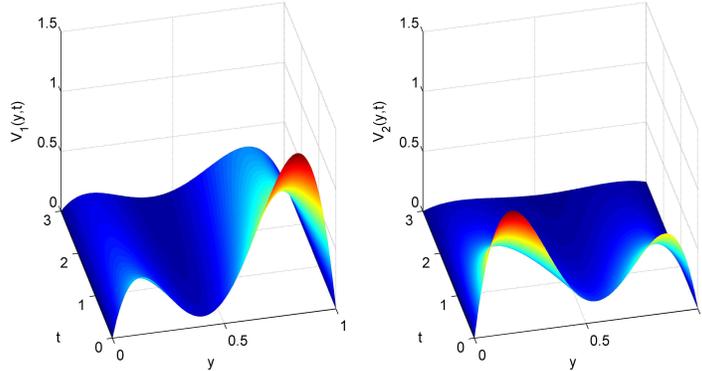}
\caption{Evolution in time of the approximate solution in the fixed
boundary problem for $v_{1}$ (left) and $v_{2}$ ( right).}
\label{ex1_sol_f}
\end{figure}
The picture on the left in Figure \ref{ex1_sol_f}
illustrates the evolution in time of the solution obtained for $v_1$ in the
fixed boundary problem, and the picture on the right illustrates the
evolution in time of the solution obtained for $v_2$. This solution was
calculated with the linearised Crank-Nicolson method with approximations of degree two and $h=\delta=10^{-2}$.
\begin{figure}[!htb]
\center
\includegraphics[width=0.9\textwidth]{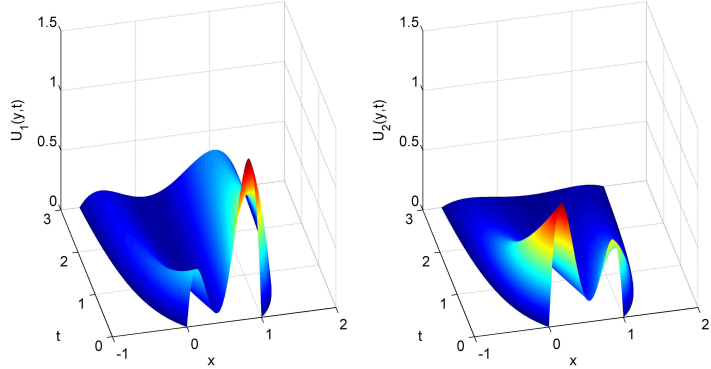}
\caption{Evolution in time of the approximate solution in the moving
boundary problem for $u_{1}$ (left) and $u_{2}$ ( right).}
\label{ex1_sol_m}
\end{figure}
The pictures in Figure \ref{ex1_sol_m} represent the solutions obtained in
the moving boundary domain, after applying the inverse transformation $\tau
^{-1}(y,t)$. If, for example, $u$ and $v$ represent the density of two populations
of bacteria, we observe that, initially, each population is concentrated mainly in two
regions and, as time increases, the two populations decrease and spread out
in the domain, as expected.

\begin{figure}[!htb]
\centering
\includegraphics[width=0.45\textwidth]{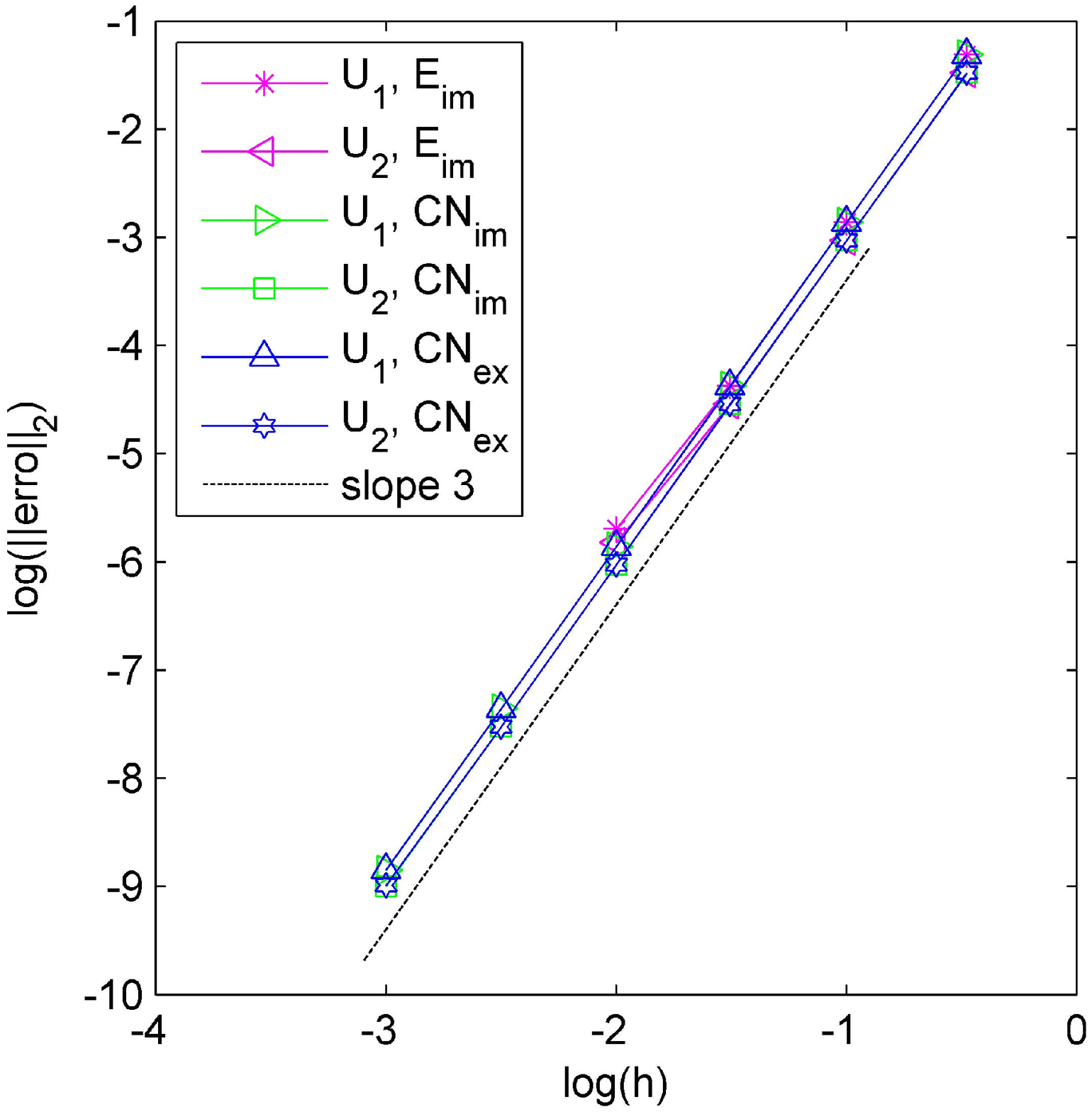} \hfill
\includegraphics[width=0.45\textwidth]{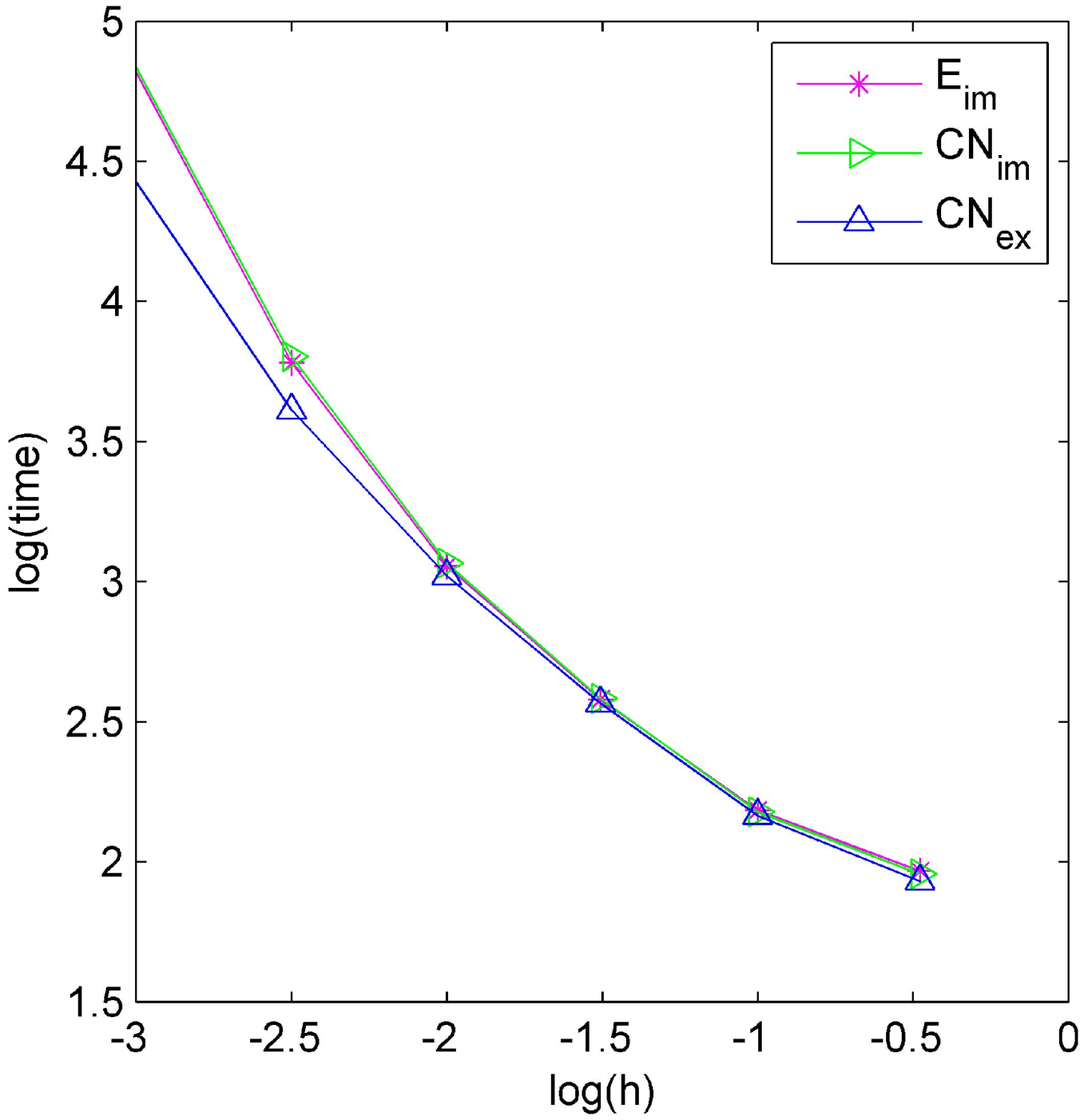}
\caption{Study of the convergence for $h$ with approximations of degree 2.}
\label{ex1_O_h}
\end{figure}

\begin{figure}[!htb]
\centering
\includegraphics[width=0.45\textwidth]{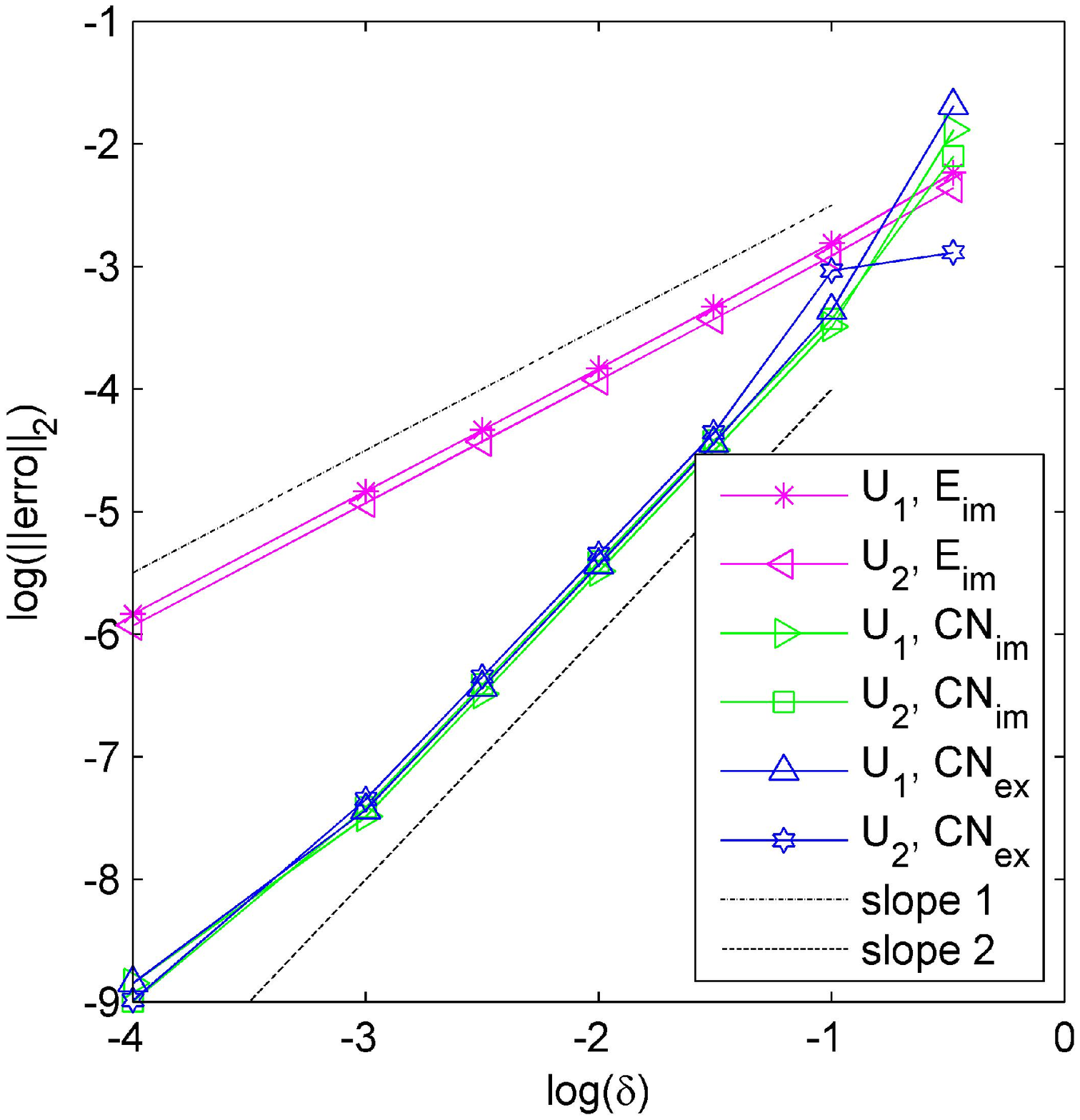} \hfill
\includegraphics[width=0.45\textwidth]{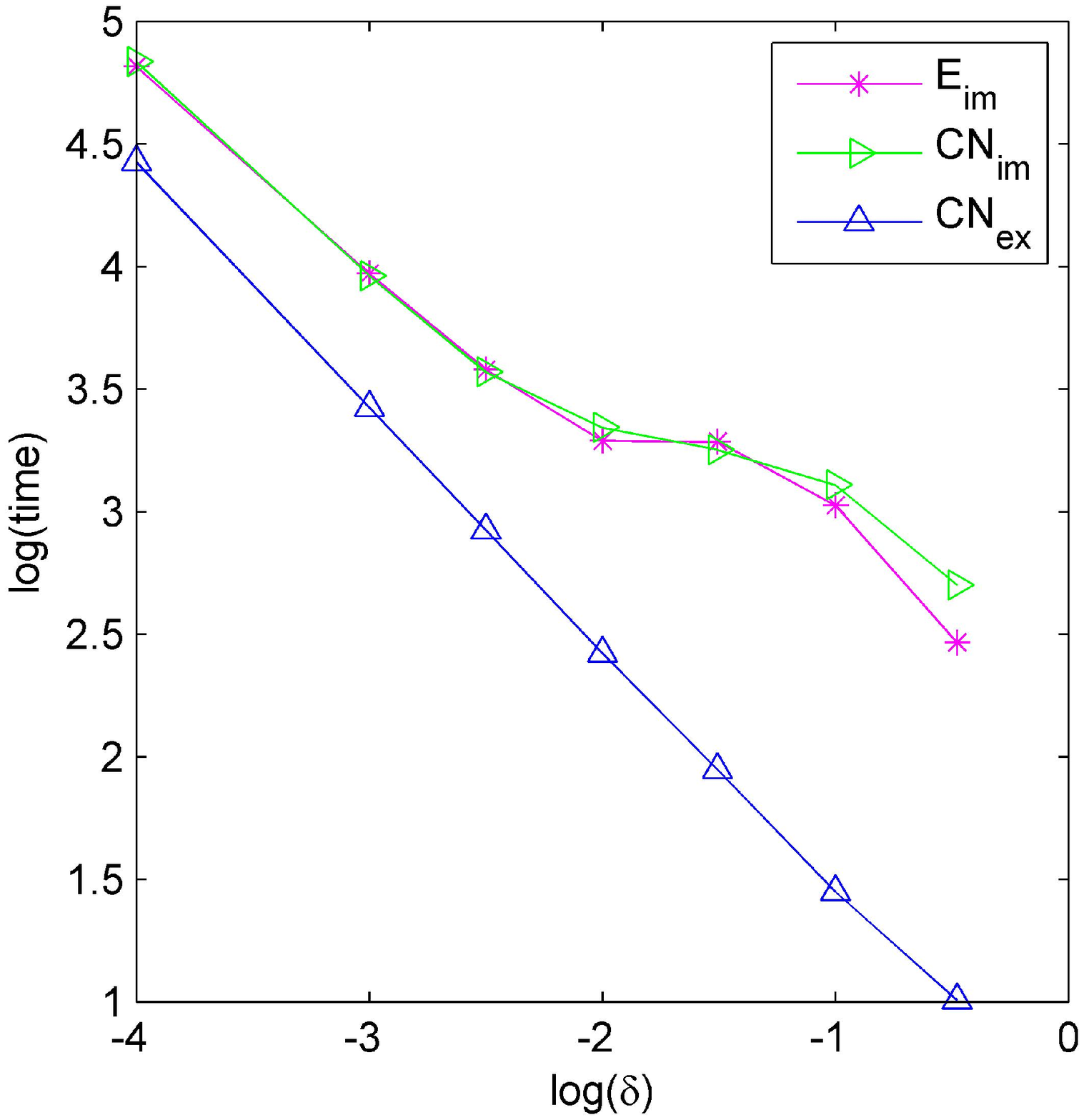}
\caption{Study of the convergence for  $\protect\delta $.}
\label{ex1_O_d}
\end{figure}

In order to analyse the convergence rates, this problem was simulated with
different combinations of $k$, $h$ and $\delta $ for each method and the error was calculated at $t=T$ and using the $L_{2}(\alpha (T),\beta (T))$-norm in the space variable. In the picture on the left in Figure \ref{ex1_O_h},  the logarithms of the errors versus the logarithm of $h$ for the simulations with $\delta
=10^{-4}$ and approximations of degree $2$, are represented. As expected, the order of convergence is approximately 3, as was proved in Theorem \ref{conv_h}. In the picture on the right we plotted the logarithms of the CPU time versus the logarithm of $h$. For large $h$, the three methods took approximately the same time, but as $h$ decreases the implicit methods take more time than the explicit one.\\
The logarithms of the errors versus the logarithm of $\delta $ for
the simulations with $h=10^{-3}$ and approximations of degree $2$, are
represented in the picture on the left in Figure \ref{ex1_O_d}. The results are in
accordance with the orders of convergence proved in Theorems \ref{conv_E_im}, \ref{conv_CN_im} and \ref{conv_CN_ex}.\\
The logarithms of the CPU time versus the logarithm of $\delta$ are plotted in the picture on the right. The implicit methods take much more time than the explicit one for big values of $\delta$, because the fixed point method requires a considerable number of iterations to obtain the predefined tolerance.\\

\begin{table}[!htb]
\begin{center}
\begin{tabular}{l|c|c|c|c|}
& \multicolumn{4}{|c|}{$\displaystyle\max_{j=1,\dots,n_p}%
\{|u_1(P_j,t_i)-U_1^{(i)}(P_j)|\}$} \\ \hline
$t_i$ & MFEM\cite{RACF14} & $E_{im}$& $CN_{im}$& $CN_{ex}$ \\ \hline
$0.001$ & 7.30e-08 & 1.26e-07 & 5.17e-10  & 2.65e-10  \\
$0.005$ & 8.95e-08 & 5.25e-07 & 1.56e-09  & 1.03e-09  \\
$0.01$  & 2.79e-08 & 8.65e-07 & 2.14e-09  & 1.46e-09  \\
$0.02$  & 1.33e-08 & 1.33e-06 & 2.59e-09  & 1.84e-09  \\
$0.05$  & 7.27e-08 & 2.09e-06 & 2.73e-09  & 2.05e-09  \\
$0.5$   & 1.90e-08 & 2.49e-06 & 1.04e-09  & 1.06e-09  \\
$1$     & 2.12e-08 & 1.51e-06 & 4.43e-10  & 5.06e-10  \\ \hline
\end{tabular}\\
\begin{tabular}{l|c|c|c|c|}
& \multicolumn{4}{|c|}{$\displaystyle%
\max_{j=1,\dots,n_p}\{|u_2(P_j,t_i)-U_2^{(i)}(P_j)|\}$} \\ \hline
$t_i$ & MFEM\cite{RACF14} & $E_{im}$& $CN_{im}$& $CN_{ex}$ \\ \hline
$0.001$  & 4.25e-08 &  4.7e-08   & 2.26e-10  & 6.36e-10 \\
$0.005$  & 5.20e-08 &  1.94e-07  & 5.78e-10  & 1.45e-09 \\
$0.01$   & 1.62e-08 &  3.25e-07  & 7.58e-10  & 1.80e-09 \\
$0.02$   & 7.74e-09 &  5.20e-07  & 9.01e-10  & 2.02e-09 \\
$0.05$   & 4.22e-08 &  8.84e-07  & 1.05e-09  & 2.16e-09 \\
$0.5$    & 1.07e-08 &  1.45e-06  & 8.00e-10  & 1.09e-09 \\
$1$      & 9.33e-09 &  1.13e-06  & 4.84e-10  & 5.59e-10 \\ \hline
\end{tabular}%
\end{center}
\caption{Comparison of the present method with the moving finite element
method in \protect\cite{RACF14}}
\label{tabela1}
\end{table}

In Table \ref{tabela1}, we compare the error
of the present method with the error of the moving finite element method
presented in \cite{RACF14}. All the simulations were done with approximations
of degree five and four finite elements. We used $\delta =10^{-4}$ for the
present methods and $10^{-10}$ for the integrator's error tolerance in the
moving finite element method.


\section{Conclusions}

We established sufficient conditions on the data to obtain optimal convergence rates for some finite element solutions with piecewise polynomial of
arbitrary degree basis functions in space when applied to a system of
nonlocal parabolic equations. Some numerical experiments were presented,
considering different time integrators. The
numerical results are in accordance with the theoretical results and are
similar in accuracy to results obtained by other method.

\section*{Acknowledgements}

This work was partially supported by the research projects:\newline
UID/MAT/00212/2013, financed by FEDER
through the - Programa Operacional Factores de Competitividade, FCT -
Fundação para a Ciência e a Tecnologia and CAPES - Brazil, Grant BEX 2478-12-9.


\begin{thebibliography}{10}

\bibitem{AK00}
Azmy~S. Ackleh and Lan Ke.
\newblock Existence-uniqueness and long time behavior for a class of nonlocal
  nonlinear parabolic evolution equations.
\newblock {\em Proc. Amer. Math. Soc.}, 128(12):3483--3492, 2000.

\bibitem{ADFR14}
Rui M.~P. Almeida, José C.~M. Duque, Jorge Ferreira, and Rui~J. Robalo.
\newblock The {C}rank-{N}icolson-{G}alerkin finite element method for a
  nonlocal parabolic equation with moving boundaries.
\newblock {\em Numerical Methods for Partial Differential Equations}, 2014.
\newblock doi: 10.1002/num.21957.

\bibitem{fer99}
Rachid Benabidallah and Jorge Ferreira.
\newblock On hyperbolic-parabolic equations with nonlinearity of
  {K}irchhoff-{C}arrier type in domains with moving boundary.
\newblock {\em Nonlinear Anal.}, 37(3, Ser. A: Theory Methods):269--287, 1999.

\bibitem{BS09}
Mostafa Bendahmane and Mauricio~A. Sep{\'u}lveda.
\newblock Convergence of a finite volume scheme for nonlocal reaction-diffusion
  systems modelling an epidemic disease.
\newblock {\em Discrete Contin. Dyn. Syst. Ser. B}, 11(4):823--853, 2009.

\bibitem{BG04}
Daniele Boffi and Lucia Gastaldi.
\newblock Stability and geometric conservation laws for {ALE} formulations.
\newblock {\em Comput. Methods Appl. Mech. Engrg.}, 193(42-44):4717--4739,
  2004.

\bibitem{bri07}
A.~C. Briozzo, M.~F. Natale, and D.~A. Tarzia.
\newblock Explicit solutions for a two-phase unidimensional
  {L}am\'e-{C}lapeyron-{S}tefan problem with source terms in both phases.
\newblock {\em J. Math. Anal. Appl.}, 329(1):145--162, 2007.

\bibitem{fer03}
M.~M. Cavalcanti, V.~N. Domingos~Cavalcanti, J.~Ferreira, and R.~Benabidallah.
\newblock On global solvability and asymptotic behaviour of a mixed problem for
  a nonlinear degenerate {K}irchhoff model in moving domains.
\newblock {\em Bull. Belg. Math. Soc. Simon Stevin}, 10(2):179--196, 2003.

\bibitem{CC03a}
N.-H. Chang and M.~Chipot.
\newblock Nonlinear nonlocal evolution problems.
\newblock {\em RACSAM. Rev. R. Acad. Cienc. Exactas F\'\i s. Nat. Ser. A Mat.},
  97(3):423--445, 2003.

\bibitem{chi00}
M.~Chipot.
\newblock {\em Elements of nonlinear analysis}.
\newblock Birkh\"auser Advanced Texts: Basler Lehrb\"ucher. [Birkh\"auser
  Advanced Texts: Basel Textbooks]. Birkh\"auser Verlag, Basel, 2000.

\bibitem{CL97}
M.~Chipot and B.~Lovat.
\newblock Some remarks on nonlocal elliptic and parabolic problems.
\newblock In {\em Proceedings of the {S}econd {W}orld {C}ongress of {N}onlinear
  {A}nalysts, {P}art 7 ({A}thens, 1996)}, volume~30, pages 4619--4627, 1997.

\bibitem{CL99}
M.~Chipot and B.~Lovat.
\newblock On the asymptotic behaviour of some nonlocal problems.
\newblock {\em Positivity}, 3(1):65--81, 1999.

\bibitem{CS03}
M.~Chipot and M.~Siegwart.
\newblock On the asymptotic behaviour of some nonlocal mixed boundary value
  problems.
\newblock In {\em Nonlinear analysis and applications: to {V}. {L}akshmikantham
  on his 80th birthday. {V}ol. 1, 2}, pages 431--449. Kluwer Acad. Publ.,
  Dordrecht, 2003.

\bibitem{CVC03}
M.~Chipot, V.~Valente, and G.~Vergara~Caffarelli.
\newblock Remarks on a nonlocal problem involving the {D}irichlet energy.
\newblock {\em Rend. Sem. Mat. Univ. Padova}, 110:199--220, 2003.

\bibitem{CM01}
Michel Chipot and Luc Molinet.
\newblock Asymptotic behaviour of some nonlocal diffusion problems.
\newblock {\em Appl. Anal.}, 80(3-4):279--315, 2001.

\bibitem{CMF04}
F.~J. S.~A. Corr{\^e}a, Silvano D.~B. Menezes, and J.~Ferreira.
\newblock On a class of problems involving a nonlocal operator.
\newblock {\em Appl. Math. Comput.}, 147(2):475--489, 2004.

\bibitem{DAAFppa}
Jos{\'e} C.~M. Duque, Rui M.~P. Almeida, Stanislav~N. Antontsev, and Jorge
  Ferreira.
\newblock A reaction-diffusion model for the nonlinear coupled system:
  existence, uniqueness, long time behavior and localization properties of
  solutions.
\newblock Available from:
  http://ptmat.fc.ul.pt/arquivo/docs/preprints/pdf/2013/preprint\_2013\_08\_An%
tontsev.pdf, 2013.

\bibitem{DAAF15}
Jos\'e~C.M. Duque, Rui~M.P. Almeida, Stanislav~N. Antontsev, and Jorge
  Ferreira.
\newblock The {E}uler {G}alerkin finite element method for a nonlocal coupled
  system of reaction-diffusion type.
\newblock {\em Journal of Computational and Applied Mathematics}, 2015.

\bibitem{EGHM02}
Robert Eymard, Thierry Gallou{\"e}t, Rapha{\`e}le Herbin, and Anthony Michel.
\newblock Convergence of a finite volume scheme for nonlinear degenerate
  parabolic equations.
\newblock {\em Numer. Math.}, 92(1):41--82, 2002.

\bibitem{fer97}
Jorge Ferreira and Nickolai~A. Lar$'$kin.
\newblock Decay of solutions of nonlinear hyperbolic-parabolic equations in
  noncylindrical domains.
\newblock {\em Commun. Appl. Anal.}, 1(1):75--81, 1997.

\bibitem{RSVPS08}
Carlos~Alberto Raposo, Mauricio Sep{\'u}lveda, Octavio~Vera Villagr{\'a}n,
  Ducival~Carvallo Pereira, and Mauro~Lima Santos.
\newblock Solution and asymptotic behaviour for a nonlocal coupled system of
  reaction-diffusion.
\newblock {\em Acta Appl. Math.}, 102(1):37--56, 2008.

\bibitem{RACF14}
Rui~J. Robalo, Rui~M.P. Almeida, Maria do~Carmo~Coimbra, and Jorge Ferreira.
\newblock A reaction-diffusion model for a class of nonlinear parabolic
  equations with moving boundaries: Existence, uniqueness, exponential decay
  and simulation.
\newblock {\em Applied Mathematical Modelling}, 38(23):5609 -- 5622, 2014.

\bibitem{fer05}
M.~L. Santos, J.~Ferreira, and C.~A. Raposo.
\newblock Existence and uniform decay for a nonlinear beam equation with
  nonlinearity of {K}irchhoff type in domains with moving boundary.
\newblock {\em Abstr. Appl. Anal.}, 2005(8):901--919, 2005.

\bibitem{san05}
M.~L. Santos, M.~P.~C. Rocha, and J.~Ferreira.
\newblock On a nonlinear coupled system for the beam equations with memory in
  noncylindrical domains.
\newblock {\em Asymptot. Anal.}, 45(1-2):113--132, 2005.

\bibitem{Tho06}
Vidar Thom{\'e}e.
\newblock {\em Galerkin finite element methods for parabolic problems},
  volume~25 of {\em Springer Series in Computational Mathematics}.
\newblock Springer-Verlag, Berlin, second edition, 2006.

\bibitem{ZC05}
S.~Zheng and M.~Chipot.
\newblock Asymptotic behavior of solutions to nonlinear parabolic equations
  with nonlocal terms.
\newblock {\em Asymptot. Anal.}, 45(3-4):301--312, 2005.

\end{thebibliography}
\end{document}